\documentclass[11pt,twoside,a4paper]{article}
\usepackage{color}
\usepackage{amscd}
\usepackage{amsmath}
\usepackage{amssymb}
\usepackage{latexsym}
\usepackage{stmaryrd}
\usepackage{pgf,tikz}
\usepackage{mathrsfs}
\usetikzlibrary{arrows}
\usepackage{shuffle}
\usepackage[latin1]{inputenc}
\usepackage[T1]{fontenc}   
\usepackage[english]{babel}
\input{xy}
\xyoption{all}

\newtheorem{definition}{\indent Definition}
\newtheorem{lemma}{\indent Lemma}
\newtheorem{corollary}{\indent Corollary}
\newtheorem{theorem}{\indent Theorem}
\newtheorem{proposition}{\indent Proposition}

\newenvironment{proof}{\textbf{Proof.}}{\hfill $\Box$}

\newcommand{\K}{\mathbb{K}}
\newcommand{\g}{\mathfrak{g}}
\newcommand{\h}{\mathcal{H}}

\newcommand{\U}{\mathcal{U}}

\newcommand{\N}{\mathbb{N}}

\newcommand{\siso}{\g_{SISO}}
\newcommand{\SISO}{G_{SISO}}
\newcommand{\Lie}{\mathcal{L}ie}
\newcommand{\Mag}{\mathcal{M}ag}
\newcommand{\shufg}[1]{\:_{#1}\shuffle\:}
\newcommand{\shufd}[1]{\:\shuffle_{#1}\:}

\begin{document}

\title{Extension of the product of a post-Lie algebra and application to the SISO feedback transformation group}
\author{Lo\"\i c Foissy\\ \\
{\small \it Fédération de Recherche Mathématique du Nord Pas de Calais FR 2956}\\
{\small \it Laboratoire de Mathématiques Pures et Appliquées Joseph Liouville}\\
{\small \it Université du Littoral Côte dOpale-Centre Universitaire de la Mi-Voix}\\ 
{\small \it 50, rue Ferdinand Buisson, CS 80699,  62228 Calais Cedex, France}\\ \\
{\small \it email: foissy@univ-littoral.fr}}

\date{}

\maketitle

\begin{abstract}
We describe the both post- and pre-Lie algebra $\g_{SISO}$ associated to the affine SISO feedback transformation group.
We show that it is a member of a family of post-Lie algebras associated to representations of a particular solvable Lie algebra.
We first construct the extension of the magmatic product of a post-Lie algebra to its enveloping algebra, 
which allows to describe free post-Lie algebras and is widely used to obtain the enveloping of $\g_{SISO}$ and its dual.
\end{abstract}



\tableofcontents

\section*{Introduction}

The affine SISO feedback transformation group  $\SISO$ \cite{Gray}, which appears in Control Theory,
can be seen as the character group of a Hopf algebra $\h_{SISO}$; let us start by a short presentation of this object
(we slightly modify the notations of \cite{Gray}).
\begin{enumerate}
\item First, let us recall some algebraic structures on noncommutative polynomials.
\begin{enumerate}
\item Let $x_1$,  $x_2$ be two indeterminates. We consider the algebra of noncommutative polynomials $\K\langle x_1,x_2\rangle$.
As a vector space, it is generated by words in letters $x_1$, $x_2$; its product is the concatenation of words;
its unit, the empty word, is denoted by $\emptyset$.
\item $\K\langle x_1,x_2\rangle$ is a Hopf algebra with the concatenation product and the deshuffling coproduct $\Delta_{\shuffle}$,
defined by $\Delta_{\shuffle}(x_i)=x_i\otimes \emptyset+\emptyset \otimes x_i$, for $i \in \{1,2\}$.
\item $\K\langle x_1,x_2\rangle$ is also a commutative, associative algebra with the shuffle product $\shuffle$: for example, if $i,j,k,l\in \{1,2\}$,
\begin{align*}
x_i\shuffle x_j&=x_ix_j+x_jx_i,\\
x_ix_j\shuffle x_k&=x_ix_jx_k+x_ix_kx_j+x_kx_ix_j,\\
x_i\shuffle x_jx_k&=x_ix_jx_k+x_jx_ix_k+x_jx_kx_i,\\
x_ix_j\shuffle x_kx_l&=x_ix_jx_kx_l+x_ix_kx_jx_l+x_ix_kx_lx_j+x_kx_ix_jx_l+x_kx_ix_lx_j+x_kx_lx_ix_j.
\end{align*} \end{enumerate}
\item The vector space $\K\langle x_1,x_2\rangle^2$ is generated by words $x_{i_1}\ldots x_{i_k}\epsilon_j$, where $k\geq 0$, $i_1,\ldots,i_k,j\in \{1,2\}$,
and $(\epsilon_1,\epsilon_2)$ denotes the canonical basis of $\K^2$. 
\item As an algebra, $\h_{SISO}$ is equal to the symmetric algebra $S(\K\langle x_1,x_2\rangle^2)$; its product is denoted by $\mu$
and its unit by $1$. Two coproducts  $\Delta_*$ and $\Delta_\bullet$ are defined on $\h_{SISO}$. For all $h\in \h_{SISO}$, we put
$\overline{\Delta}_*(h)=\Delta_*(h)-1\otimes h$ and $\overline{\Delta}_\bullet(h)=\Delta_\bullet(h)-1\otimes h$. Then:
\begin{itemize}
\item For all $i\in \{1,2\}$, $\overline{\Delta}_*(\emptyset\epsilon_i)=\emptyset\epsilon_i\otimes 1$.
\item For all $g\in \K\langle x_1,x_2\rangle$, for all $i\in  \{1,2\}$:
\begin{align*}
\overline{\Delta}_*\circ \theta_{x_1}(g\epsilon_i)&=(\theta_{x_1}\otimes Id)\circ \overline{\Delta}_*(g\epsilon_i)+
(\theta_{x_2}\otimes \mu)\circ(\overline{\Delta}_*\otimes Id)(\Delta_{\shuffle}(g) \epsilon_i\otimes \epsilon_2),\\
\overline{\Delta}_*\circ \theta_{x_2}(g\epsilon_i)&=(\theta_{x_2}\otimes \mu)\circ(\overline{\Delta}_*\otimes Id)(\Delta_{\shuffle}(g) \epsilon_i\otimes \epsilon_1),
\end{align*}
where $\theta_{x}(h\epsilon_i)=xh\epsilon_i$ for all $x\in \{x_1,x_2\}$, $h\in \K\langle x_1,x_2\rangle$, $i\in  \{1,2\}$. 
These are formulas of Lemma 4.1 of \cite{Gray}, with the notations $a_w=w\epsilon_2$, $b_w=w\epsilon_1$, $\theta_0=\theta_{x_1}$,
$\theta_1=\theta_{x_2}$ and $\tilde{\Delta}=\overline{\Delta}_*$.
\item for all $g\in \K\langle x_1,x_2\rangle$:
\begin{align*}
\overline{\Delta}_\bullet(g\epsilon_1)&=(Id\otimes \mu)\circ (\overline{\Delta}_*\otimes Id)(\Delta_{\shuffle}(g)(\epsilon_1\otimes \epsilon_1)),\\
\overline{\Delta}_\bullet(g\epsilon_2)&=\overline{\Delta}_*(g\epsilon_2)+
(Id\otimes \mu)\circ (\overline{\Delta}_*\otimes Id)(\Delta_{\shuffle}(g)(\epsilon_2\otimes \epsilon_1)).
\end{align*}\end{itemize}
This coproduct $\Delta_\bullet$ makes $\h_{SISO}$ a Hopf algebra, and $\Delta_*$ is a right coaction on this coproduct, that is to say:
\begin{align*}
(\Delta_\bullet \otimes Id)\circ \Delta_\bullet&=(Id \otimes \Delta_\bullet)\circ \Delta_\bullet,&
(\Delta_* \otimes Id)\circ \Delta_*&=(Id \otimes \Delta_\bullet)\circ \Delta_*.
\end{align*}
\item After the identification of $\emptyset\epsilon_1$ with the unit of $\h_{SISO}$, we obtain a commutative, graded and connected Hopf algebra, 
in other words the dual of an enveloping algebra $\U(\siso)$. 
\end{enumerate}
Our aim is to give a description of the underlying Lie algebra $\siso$. It turns out that it is both a pre-Lie algebra
(or a Vinberg algebra \cite{Cartier}, see \cite{Manchon} for a survey on these objects) and a post-Lie algebra \cite{MK,Vallette}: it has a Lie bracket
$\!_a[-,-]$ and two nonassociative products $*$ and $\bullet$, such that for all $x,y,z\in \siso$:
\begin{align*}
x*\:_a[y,z]&=(x*y)*z-x*(y*z)-(x*z)*y+x*(z*y),\\
\:_a[x,y]*z&=\:_a[x*z,y]+\:_a[x,y*z];
\end{align*}
\begin{align*}
(x\bullet y)\bullet z-x\bullet (y\bullet z)&=(x\bullet z)\bullet y-x\bullet (z\bullet y).
\end{align*}
The Lie bracket on $\siso$ corresponding to $G_{SISO}$ is $\:_a[-,-]_*$:
\begin{align*}
\forall x,y \in \siso,\: \:_a[x,y]_*&=\:_a[x,y]+x*y-y*x=x\bullet y-y\bullet x.
\end{align*} 

Let us be more precise on these structures. As a vector space, $\siso=\K\langle x_1,x_2\rangle^2$, and:
\begin{align*}
&\forall f,g\in \K\langle x_1,x_2\rangle, \:\forall i,j \in  \{1,2\},& \:_a[f\epsilon_i,g\epsilon_j]
=\begin{cases}
0\mbox{ if }i=j,\\
-f\shuffle g \epsilon_2\mbox{ if }i=2\mbox{ and }j=1,\\
f\shuffle g \epsilon_2\mbox{ if }i=1\mbox{ and }j=2.
\end{cases}\end{align*}
The magmatic product $*$ is inductively defined. If $f,g \in \K\langle x_1,x_2\rangle$ and $i,j \in  \{1,2\}$:
\begin{align*}
\emptyset\epsilon_i*g\epsilon_j&=0,&x_2f\epsilon_i*g\epsilon_1&=x_2(f\epsilon_i*g\epsilon_1)+x_2(f\shuffle g)\epsilon_i,\\
x_1f\epsilon_i*g\epsilon_j&=x_1(f\epsilon_i*g\epsilon_j),&x_2f\epsilon_i*g\epsilon_2&=x_2(f\epsilon_i*g\epsilon_2)+x_1(f\shuffle g)\epsilon_i.
\end{align*}
The pre-Lie product $\bullet$, first determined in \cite{Gray}, is given by: 
$$\forall f,g\in \K\langle x_1,x_2\rangle, \:\forall i,j\in  \{1,2\}, \:f\epsilon_i \bullet g\epsilon_j=(f\shuffle g)\delta_{i,1} \epsilon_j+f\epsilon_i*g\epsilon_j.$$
We shall show here that this is a special case of a family of post-Lie algebras, associated to modules over certain solvable Lie algebras.\\

We start with general preliminary results on post-Lie algebras.
We extend the now classical Oudom-Guin construction on prelie algebras \cite{Oudom2,Oudom1} to the post-Lie context in the first section:
this is a result of \cite{Ebrahimi} (Proposition 3.1), which we prove here in a different, less direct way; our proof allows also
to obtain a description of free post-Lie algebras.
Recall that if $(V,*)$ is a pre-Lie algebra, the pre-Lie product $*$ can be extended to $S(V)$ in such a way that the product defined by:
$$\forall f,g\in S(V),\: f\circledast g=\sum f*g^{(1)} g^{(2)}$$
is associative, and makes $S(V)$ a Hopf algebra, isomorphic to $\U(V)$.
For any magmatic algebra $(V,*)$, we construct in a similar way an extension of $*$ to $T(V)$ in Proposition \ref{prop2}.
We prove in Theorem \ref{theo4} that the product $\circledast$ defined by:
$$\forall f,g\in T(V),\: f\circledast g=\sum f*g^{(1)} g^{(2)}$$
makes $T(V)$ a Hopf algebra. The Lie algebra of its primitive elements, which is the free Lie algebra $\Lie(V)$ generated by $V$,
is stable under $*$ and turns out to be a post-Lie algebra (Proposition \ref{prop5}) satisfying a universal property (Theorem \ref{theo7}).
In particular, if $V$ is, as a magmatic algebra, freely generated by a subspace $W$,  
$\Lie(V)$ is the free post-Lie algebra generated by $W$ (Corollary \ref{cor8}).
Moreover, if $V=([-,-],*)$ is a post-Lie algebra, this construction goes through the quotient defining $\mathcal{U}(V, [-,-])$, 
defining a new product $\circledast$ on it, making it isomorphic to the enveloping algebra of $V$ with the Lie bracket defined by:
$$\forall x,y\in V, \:[x,y]_*=[x,y]+x*y-y*x.$$
For example, if $x_1,x_2,x_3 \in V$:
 \begin{align*}
 x_1\circledast x_2&=x_1x_2+x_1*x_2\\
 x_1\circledast x_2x_3&=x_1x_2x_3+(x_1*x_2)x_3+(x_1*x_3)x_2+(x_1*x_2)*x_3-x_1*(x_2*x_3)\\
 x_1x_2\circledast x_3&=x_1x_2x_3+(x_1*x_3)x_2+x_1(x_2*x_3).
 \end{align*}
In the particular case where $[-,-]=0$, we recover the Oudom-Guin construction.

The second section is devoted to the study of a particular solvable Lie algebra $\g_a$ associated to an element $a\in \K^N$. 
As the Lie bracket of $\g_a$ comes from an associative product, the construction of the first section holds, with many simplifications: 
we obtain an explicit description of $\U(\g_a)$ with the help of a product $\blacktriangleleft$ on $S(\g_a)$ (Proposition \ref{prop19}). 
A short study of $\g_a$-modules when $a=(1,0,\ldots,0)$ (which is a generic case) is done in Proposition \ref{prop21}, considering $\g_a$ 
as an associative algebra, and in Proposition \ref{prop22}, considering it as a Lie algebra. In particular, if $\K$ is algebraically closed,
any $\g_a$ modules inherits a natural decomposition in characteristic subspaces.

Our family of post-Lie algebras is introduced in the third section; it is reminescent of the construction of \cite{Foissy2}.
Let us fix a vector space $V$, $(a_1,\ldots,a_N)\in\K^N$ and a family  $F_1,\ldots,F_N$ of endomorphisms of $V$. We define a product $*$ on $T(V)^N$,
such that for all $f,g \in T(V)$, $x\in V$, $i,j\in \{1,\ldots,N\}$:
\begin{align*}
\emptyset\epsilon_i*g\epsilon_j&=0,\\
xf\epsilon_i*g\epsilon_j&=x(f\epsilon_i*g\epsilon_j)+F_j(x)(f\shuffle g)\epsilon_i,
\end{align*}
where $(\epsilon_1,\ldots,\epsilon_N)$ is the canonical basis of $\K^N$ and $\shuffle$ is the shuffle product of $T(V)$. 
The Lie bracket of $T(V)^N$ that we shall use here is:
$$\forall f,g \in T(V), \:\forall i,j\in \{1,\ldots,N\},\: _a[f\epsilon_i,g\epsilon_j]=(f\shuffle g)(a_i \epsilon_j-a_j\epsilon_i).$$
This Lie bracket comes from an associative product $_a\shuffle$ defined by:
$$\forall f,g \in T(V), \:\forall i,j\in \{1,\ldots,N\},f\epsilon_i\:_a\shuffle g\epsilon_j=a_i(f\shuffle g)\epsilon_j.$$
We put $\bullet=*+\:_a\shuffle$.
We prove in Theorem \ref{theo30} the equivalence of the three following conditions:
\begin{itemize}
\item $(T(V)^N,\bullet)$ is a pre-Lie algebra.
\item$(T(V)^N,\:_a[-,-],*)$ is a post-Lie algebra.
\item $F_1,\ldots,F_N$ defines a structure of $\g_a$-module on $V$.
\end{itemize}
If this holds, the construction of the first section allows to obtain two descriptions of the enveloping algebra of $\U(T(V)^N)$,
respectively coming from the post-Lie product $*$ and from the pre-Lie product $\bullet$: the extensions of $*$ and of $\bullet$ are respectively
described in Propositions \ref{prop32} and \ref{prop33}. It is shown in Proposition \ref{prop34} that the two associated
descriptions of $\U(T(V)^N)$ are equal. For $\siso$, we take $a=(1,0)$, $V=Vect(x_1,x_2)$ and:
\begin{align*}
F_1=&\left(\begin{array}{cc}
0&0\\0&1
\end{array}\right),&
F_2&=\left(\begin{array}{cc}
0&1\\0&0
\end{array}\right),
\end{align*}
which indeed define a $\g_{(1,0)}$-module. 
In order to relate this to the Hopf algebra $\h_{SISO}$ of \cite{Gray}, we need to consider the dual 
of the enveloping of $T(V)^N$. First, if $a=(1,0,\ldots,0)$, we observe that the decomposition of $V$ as a $\g_a$-module of the second section
induces a graduation of the post-Lie algebra $T(V)^N$ (Proposition \ref{prop35}), unfortunately not connected: the component of degree $0$
is $1$-dimensional, generated by $\emptyset \epsilon_1$. Forgetting this element, that is, considering the augmentation ideal
of the graded post-Lie algebra $T(V)^N$, we can dualize the product $\circledast$ of $S(T(V)^N)$ in order to obtain the coproduct of the dual Hopf algebra
in an inductive way. For $\siso$, we indeed obtain the inductive formulas of $\h_{SISO}$, finally proving that the dual Lie algebra of this Hopf algebra,
which in some sense can be exponentiated to $G_{SISO}$, is indeed post-Lie and pre-Lie. \\

\textbf{Aknowledgments.} The research leading these results was partially supported by the French National Research Agency under the reference
ANR-12-BS01-0017. \\

\textbf{Notations.} \begin{enumerate}
\item Let $\K$ be a commutative field. The canonical basis of $\K^n$ is denoted by $(\epsilon_1,\ldots,\epsilon_n)$.
\item For all $n\geq 1$, we denote by $[n]$ the set $\{1,\ldots,n\}$.
\item We shall use Sweeder's notations: if $C$ is a coalgebra and $x \in C$,
\begin{align*}
\Delta^{(1)}(x)=\Delta(x)&=\sum x^{(1)}\otimes x^{(2)},\\
\Delta^{(2)}(x)=(\Delta \otimes Id)\circ \Delta(x)&=\sum x^{(1)}\otimes x^{(2)}\otimes x^{(3)},\\
\Delta^{(3)}(x)=(\Delta \otimes Id \otimes Id) \circ (\Delta \otimes Id)\circ \Delta(x)&=\sum x^{(1)}\otimes x^{(2)}\otimes x^{(3)}\otimes x^{(4)},\\
&\vdots
\end{align*}\end{enumerate}

\section{Extension of a post-Lie product}

We first generalize the Oudom-Guin extension of a pre-Lie product in a post-Lie algebraic context, as done in \cite{Ebrahimi}. 
Let us first recall what a post-Lie algebra is.

\begin{definition}\begin{enumerate}
\item A (right) \emph{post-Lie algebra} is a family $(\g,\{-,-\},*)$, where $\g$ is a vector space, $\{-,-\}$ and $*$ are bilinear products on $\g$ such that:
\begin{itemize}
\item $(\g,\{-,-\})$ is a Lie algebra.
\item For all $x,y,z \in \g$:
\begin{align}
\label{EQ1} x*\{y,z\}&=(x*y)*z-x*(y*z)-(x*z)*y+x*(z*y),\\
\label{EQ2}\{x,y\}*z&=\{x*z,y\}+\{x,y*z\}.
\end{align}\end{itemize}
\item If $(\g,\{-,-\},*)$ is post-Lie, we define a second Lie bracket on $\g$:
$$\forall x,y\in \g,\: \{x,y\}_*=\{x,y\}+x*y-y*x.$$
\end{enumerate}\end{definition}

Note that if $\{-,-\}$ is $0$, then $(\g,*)$ is a (right) pre-Lie algebra, that is to say:
\begin{align}
\forall x,y,z\in \g,\: (x*y)*z-x*(y*z)=(x*z)*y-x*(z*y).
\end{align}

%\begin{proof} The bracket $\{-,-\}_*$ is obviously antisymmetric. Let $x,y,z\in \g$.
%\begin{align*}
%&\{\{x,y\}_*,z\}_*+\{\{y,z\}_*,x\}_*+\{\{z,x\}_*,y\}_*\\
%&=\{\{x,y\},z\}+\{\{y,z\},x\}+\{\{z,x\},y\}\\
%&+\{x*y,z\}-\{y*x,z\}+\{y*z,x\}-\{z*y,x\}+\{z*x,y\}-\{x*z,y\}\\
%&+\{x,y\}*z-z*\{x,y\}+\{y,z\}*x-x*\{y,z\}+\{z,x\}*y-y*\{z,x\}\\
%&+(x*y)*z-(y*x)*z-z*(x*y)+z*(y*x)\\
%&+(y*z)*x-(z*y)*x-x*(y*z)+x*(z*y)\\
%&+(z*x)*y-(x*z)*y-y*(z*x)+y*(x*z)\\
%&=\{\{x,y\},z\}+\{\{y,z\},x\}+\{\{z,x\},y\}\\
%&-\{z,x*y\}-\{z*y,x\}+\{z,x\}*y\\
%&-\{y*x,z\}-\{y,z*x\}+\{y,z\}*x\\
%&-\{x,y*z\}-\{x*z,y\}+\{x,y\}*z\\
%&-x*\{y,z\}+(x*y)*z-x*(y*z)-(x*z)*y+x*(z*y)\\
%&-z*\{x,y\}-z*(x*y)+z*(y*x)+(z*x)*y-(z*y)*x\\
%&-y*\{z,x\}-(y*x)*z+(y*z)*x-(x*z)*y-y*(z*x)+y*(x*z)\\
%&=0.
%\end{align*}
%So the Jacobi relation is satisfied. \end{proof}

\subsection{Extension of a magmatic product}

Let $V$ be a vector space. We here use the tensor Hopf algebra $T(V)$. 
In this section, we shall denote the unit of $T(V)$ by $1$. 
Its product is the concatenation of words, and its coproduct $\Delta_{\shuffle}$ is the cocommutative deshuffling coproduct.
For example, if $x_1,x_2,x_3\in V$:
\begin{align*}
\Delta_{\shuffle}(x_1)&=x_1\otimes 1+1\otimes x_1,\\
\Delta_{\shuffle}(x_1x_2)&=x_1x_2\otimes 1+x_1\otimes x_2+x_2\otimes x_1+1\otimes x_1x_2,\\
\Delta_{\shuffle}(x_1x_2)&=x_1x_2x_3\otimes 1+x_1x_2\otimes x_3+x_1x_3\otimes x_2+x_2x_3\otimes x_1\\
&+x_1\otimes x_2x_3+x_2\otimes x_1x_3+x_3\otimes x_1x_2+1\otimes x_1x_2x_3.
\end{align*}
Its counit is denoted by $\varepsilon$: $\varepsilon(1)=1$ and if $k\geq 1$ and $x_1,\ldots,x_k\in V$, $\varepsilon(x_1\ldots x_k)=0$.

\begin{proposition}\label{prop2}
Let $V$ be a vector space and $*:V \otimes V \longrightarrow V$ be a magmatic product on $V$. 
Then $*$ can be uniquely extended as a map from $T(V) \otimes T(V)$ to $T(V)$ such that for all $f,g,h\in T(V)$, $x,y\in V$:
\begin{itemize}
\item $f*1=f$.
\item $1*f=\varepsilon(f)1$.
\item $x*(fy)=(x*f)*y-x*(f*y)$.
\item $(fg)*h=\sum \left(f*h^{(1)}\right)\left(g*h^{(2)}\right)$.
\end{itemize}\end{proposition}

\begin{proof} \textit{Existence}. We first inductively extend $*$ from $V \otimes T(V)$ to $V$. If $n\geq 0$, $x$, $y_1,\ldots,y_n \in V$, we put:
$$x*y_1\ldots y_n=\begin{cases}x\mbox{ if }n=0,\\
x*y_1\mbox{ if }n=1,\\
\displaystyle \underbrace{\underbrace{(x*(y_1\ldots y_{n-1})}_{\in V}*
\underbrace{y_n}_{\in V}}_{\in V}-\sum_{i=1}^{n-1} \underbrace{x*(\underbrace{y_1 \ldots (y_i*y_n)\ldots y_{n-1})}
_{\in V^{\otimes (n-1)}}}_{\in V}\mbox{ if }n\geq 2.
\end{cases}$$
This product is then extended from $T(V) \otimes T(V)$ to $T(V)$ in the following way: 
\begin{itemize}
\item For all $f\in T(V)$, $1*f= \varepsilon(f)1$.
\item For all $n\geq 1$, for all $x_1,\ldots,x_n \in V$, $f\in T(V)$:
$$(x_1\ldots x_n )*f=\sum (\underbrace{x_1*f^{(1)}}_{\in V}) \ldots (\underbrace{x_n*f^{(n)}}_{\in V})\in V^{\otimes n}.$$
\end{itemize}
Note that for all $n \geq 0$, $V^{\otimes n}*T(V) \subseteq V^{\otimes n}$, which induces the second point. 
Let us prove the first point with $f=x_1\ldots x_n \in V^{\otimes n}$. If $n=0$, $f*1=1*1=\varepsilon(1)1=1=f$. If $n=1$, $f\in V$, 
so $f*1=f$ by definition of the extension of $*$ on $V \otimes T(V)$. If $n\geq 2$:
$$f*1=(x_1\ldots x_n)*1=(x_1*1)\ldots (x_n*1)=x_1\ldots x_n=f.$$
Let us prove the third point for $f=y_1\ldots y_n$. Then:
$$x*(fy)=(x*f)*y-\sum x*(y_1\ldots (y_i*y)\ldots y_n).$$
Moreover, as $\Delta_{\shuffle}(y)=y\otimes 1+1\otimes y$:
$$f*y=\sum_{i=1}^n (y_1*1)\ldots (y_i*y)\ldots (y_n*1)
=\sum_{i=1}^n y_1\ldots (y_i*y)\ldots y_n.$$
So $x*(fy)=(x*f)*y-x*(f*y)$. 
Let us finally prove the last point for $f=x_1\ldots x_k$ and $g=x_{k+1}\ldots x_{k+l}$. Then:
\begin{align*}
(fg)*h&=\sum \left(x_1*h^{(1)}\right)\ldots \left(x_{k+l}*h^{(k+l)}\right)\\
&= \sum \left(x_1*(h^{(1)})^{(1)}\right) \ldots \left(x_1*(h^{(1)})^{(k)}\right)
\left(x_{k+1}*(h^{(2)})^{(1)}\right) \ldots \left(x_{k+l}*(h^{(2)})^{(l)}\right)\\
&=\sum \left((x_1\ldots x_k)*h^{(1)}\right) \left((x_{k+1}\ldots x_{k+l})*h^{(2)}\right)\\
&=\sum \left(f*h^{(1)}\right)\left(g*h^{(2)}\right).
\end{align*}
We used the coassociativity of $\Delta_{\shuffle}$ for the second equality. \\

\textit{Unicity}. The first and third points uniquely determine $x*(x_1\ldots x_n)$ for $x,x_1,\ldots, x_n \in V$, by induction on $n$;
the second and fourth points then uniquely determine $f*(x_1\ldots x_n)$ for all $f\in T(V)$ by induction on the length of $f$. \end{proof}\\

\textbf{Examples.} If $x_1,x_2,x_3,x_4 \in V$:
\begin{align*}
(x_1x_2)*x_3&=(x_1*x_3) x_2+x_1(x_2*x_3),\\
x_1*(x_2x_3)&=(x_1*x_2)*x_3-x_1*(x_2*x_3),\\
(x_1x_2x_3)*x_4&=(x_1*x_4)x_2x_3+x_1(x_2*x_4)x_3+x_1x_2(x_3*x_4),\\
(x_1x_2)*(x_3x_4)&=((x_1*x_3)*x_4)x_2-(x_1*(x_3*x_4))x_2+x_1((x_2*x_3)*x_4),\\
&-x_1(x_2*(x_3*x_4))+(x_1*x_3)(x_2*x_4)+(x_1*x_4)(x_2*x_3),\\
x_1*(x_2x_3x_4)&=((x_1*x_2)*x_3)*x_4-(x_1*(x_2*x_3))*x_4-(x_1*(x_2*x_4))*x_3\\
&+x_1*((x_2*x_4)*x_3)-(x_1*x_2)*(x_3*x_4)+x_1*(x_2*(x_3*x_4)).
\end{align*}

\begin{lemma}\label{lem3} \begin{enumerate}
\item For all $k \in \mathbb{N}$, $V^{\otimes k}*T(V) \subseteq V^{\otimes k}$.
\item For all $f,g\in T(V)$, $\varepsilon (f*g)=\varepsilon(f)\varepsilon(g)$.
\item For all $f,g \in T(V)$, $\Delta_{\shuffle}(f*g)=\Delta_{\shuffle}(f)*\Delta_{\shuffle}(g)$.
\item For all $f,g \in T(V)$, $y\in V$, $f*(gy)=(f*g)*y-f*(g*y)$.
\item For all $f,g,h \in T(V)$, $(f*g)*h=\sum f*\left(\left(g*h^{(1)}\right) h^{(2)}\right)$.
\end{enumerate}\end{lemma}

\begin{proof} 1. This was observed in the proof of Proposition \ref{prop2}.\\

2. From the first point, $Ker(\varepsilon)*T(V)+T(V)*Ker(\varepsilon)\subseteq Ker(\varepsilon)$,
so if $\varepsilon(f)=0$ or $\varepsilon(g)=0$, then $\varepsilon(f*g)=0$. As $\varepsilon(1* 1)=1$, the second point holds for all $f,g$. \\

3. We prove it for $f=x_1\ldots x_n$, by induction on $n$. If $n=0$, then $f=1$. Moreover, 
$\Delta_{\shuffle}(1*g)=\varepsilon(g) \Delta_{\shuffle}(1)=\varepsilon(g) 1\otimes 1$, and:
$$\Delta_{\shuffle}(f)*\Delta_{\shuffle}(g)=\sum 1*g^{(1)}\otimes 1*g^{(2)}=\varepsilon\left(g^{(1)}\right)\varepsilon\left(g^{(2)}\right)1\otimes 1=\varepsilon(g)1\otimes 1.$$
If $n=1$, then $f\in V$. In this case, from the second point, $f*g \in V$, so $\Delta_{\shuffle}(f*g)=f*g \otimes 1+1\otimes f*g$. Moreover:
\begin{align*}
\Delta_{\shuffle}(f)*\Delta_{\shuffle}(g)&=(f\otimes 1+1\otimes f)*\Delta_{\shuffle}(g)\\
&=\sum f*g^{(1)} \otimes 1*g^{(2)}+\sum 1*g^{(1)}\otimes f*g^{(2)}\\
&=\sum f*g^{(1)} \otimes  \varepsilon\left(g^{(2)}\right)1+\sum \varepsilon\left(g^{(1)}\right)1\otimes f*g^{(2)}\\
&=f*g\otimes 1+1\otimes f*g.
\end{align*}
If $n\geq 2$, we put $f_1=x_1\ldots x_{n-1}$ and $f_2=x_n$.  By the induction hypothesis applied to $f_1$:
\begin{align*}
\Delta_{\shuffle}(f*g)&=\sum \Delta_{\shuffle}\left(\left(f_1*g^{(1)}\right)\left (f_2*g^{(2)}\right)\right)\\
&=\Delta_{\shuffle}\left(f_1*g^{(1)}\right)\Delta_{\shuffle}\left(f_2*g^{(2)}\right)\\
&=\sum \left(f_1^{(1)}*(g^{(1)})^{(1)}\right) \left(f_2^{(1)}*(g^{(2)})^{(1)}\right)
 \otimes \left(f_1^{(2)}*(g^{(1)})^{(2)}\right) \left(f_2^{(2)}*(g^{(2)})^{(2)}\right) \\
%&=\sum \left(f_1^{(1)}*(g^{(1)})^{(1)}\right) \left(f_2^{(1)}*(g^{(1)})^{(2)}\right) 
%\otimes \left(f_1^{(2)}*(g^{(2)})^{(1)}\right) \left(f_2^{(2)}*(g^{(2)})^{(2)}\right) \\
&=\sum (f_1f_2)^{(1)}*g^{(1)} \otimes (f_1f_2)^{(2)}*g^{(2)}\\
&=\Delta_{\shuffle}(f)*\Delta_{\shuffle}(g).
\end{align*}
We used the cocommutativity of $\Delta_{\shuffle}$ for the fourth equality.\\

4. We prove it for $f=x_1\ldots x_n$, by induction on $n$. If $n=0$, then $f=1$ and:
$$1*(gy)=(1*g)*y-1*(g*y)=\varepsilon(g)\varepsilon(y)-\varepsilon(g*y)=0.$$
For $n=1$, this comes immediately from Proposition \ref{prop2}-3. 
If $n \geq 2$, we put $f_1=x_1\ldots x_{n-1}$ and $f_2=x_n$. The induction hypothesis holds for $f_1$. Moreover:
\begin{align*}
f*(gy)&=\sum \left(f_1*g^{(1)}\right)\left(f_2*\left(g^{(2)}y\right)\right)+\sum\left(f_1*\left(g^{(1)}y\right)\right)\left(f_2*g^{(2)}\right)\\
&=\sum \left(f_1*g^{(1)}\right)\left(\left(f_2*g^{(2)}\right)*y\right)-\sum\left(f_1*g^{(1)}\right)\left(f_2*\left(g^{(2)}*y\right)\right)\\
&+\sum\left(\left(f_1*g^{(1)}\right)*y\right)\left(f_2*g^{(2)}\right)-\sum\left(f_1*\left(g^{(1)}*y\right)\right)\left(f_2*g^{(2)}\right), \\ 
\left(f*g\right)*y&=\sum\left(\left(f_1*g^{(1)}\right)\left(f_2*g^{(2)}\right)\right)*y\\
&=\sum\left(\left(f_1*g^{(1)}\right)*y\right)\left(f_2*g^{(2)}\right)+\sum\left(f_1*g^{(1)}\right)\left(\left(f_2*g^{(2)}\right)*y\right),\\ 
f*(g*y)&=\sum\left(f_1*\left(g*y\right)^{(1)}\right)\left(f_2*\left(g*y\right)^{(2)}\right)\\
&=\sum\left(f_1*\left(g^{(1)}*y\right)\right)\left(f_2*g^{(2)}\right)+\sum\left(f_1*g^{(1)}\right)\left(f_2*\left(g^{(2)}*y\right)\right).
\end{align*}
We use the third point for the third computation. So the result holds for all $f$.\\

5. We prove this for $h=z_1\ldots z_n$ and we proceed by induction on $n$. If $n=0$, then $h=1$ and
$(f*g)*1=f*g$. Moreover, $\sum f*\left(\left(g*h^{(1)}\right) h^{(2)}\right)=f*((g*1) 1)=(f*g) 1=f*g$.
If $n=1$, then $h \in V$, so $\Delta_{\shuffle}(h)=h\otimes 1+1\otimes h$. So:
\begin{align*}
\sum f*\left(\left(g*h^{(1)}\right)h^{(2)}\right)&=f*((g*h)1)+f*((g*1)h)\\
&=f*(g*h)+f*(gh)\\
&=f*(g*h)+(f*g)*h-f*(g*h)\\
&=(f*g)*h.
\end{align*}
We use Proposition \ref{prop2}-3 for the third equality. If $n \geq 2$, we put $h_1=z_1\ldots z_{n-1}$ and $h_2=z_n$.
From the fourth point:
\begin{align*}
(f*g)*h&=((f*g)*h_1)*h_2-(f*g)*(h_1*h_2)\\
&=\sum \left(f*\left( \left(g*h_1^{(1)}\right)h_1^{(2)}\right)\right)*h_2-\sum f*\left(\left(g*\left(h_1*h_2\right)^{(1)}\right) \left(h_1*h_2\right)^{(2)}\right)\\
&=\sum f*\left(\left(\left(g*h_1^{(1)}\right)h_1^{(2)}\right)*h_2\right)+\sum f*\left(\left(g*h_1^{(1)}\right)h_1^{(2)} h_2\right)\\
&-\sum f*\left(\left(g*\left(h_1^{(1)}*h_2^{(1)}\right)\right) \left(h_1^{(2)}*h_2^{(2)}\right)\right)\\
&=\sum f*\left(\left(\left(g*h_1^{(1)}\right)*h_2\right)h_1^{(2)}\right)+\sum f*\left(\left(g*h_1^{(1)}\right)\left(h_1^{(2)}*h_2\right)\right)\\
&+\sum f*\left(\left(g*h_1^{(1)}\right)h_1^{(2)} h_2\right)-\sum f*\left(\left(g*\left(h_1^{(1)}*h_2\right)\right) h_1^{(2)} \right)\\
&-\sum f*\left(\left(g*h_1^{(1)} \right) \left(h_1^{(2)}*h_2\right)\right)\\
&=\sum f*\left(\left(g*\left(h_1^{(1)}*h_2\right)\right)h_1^{(2)}\right)+\sum f*\left(\left(g*\left(h_1^{(1)} h_2\right)\right)h_1^{(2)}\right)\\
&+\sum f*\left(\left(g*h_1^{(1)}\right)\left(h_1^{(2)}*h_2\right)\right)+\sum f*\left(\left(g*h_1^{(1)}\right)h_1^{(2)} h_2\right)\\
&-\sum f*\left(\left(g*\left(h_1^{(1)}*h_2\right)\right) h_1^{(2)} \right)-\sum f*\left(\left(g*h_1^{(1)} \right) \left(h_1^{(2)}*h_2\right)\right)\\
&=\sum f*\left(\left(g*\left(h_1^{(1)} h_2\right)\right)h_1^{(2)}\right)+\sum f*\left(\left(g*h_1^{(1)}\right)h_1^{(2)} h_2\right).
\end{align*}
For the second equality, we used the induction hypothesis on $h_1$ and $h_1*h_2 \in V^{\otimes (k-1)}$ by the first point;
we used the third point for the third equality. As $\Delta_{\shuffle}(h_2)=h_2\otimes 1+1\otimes h_2$, 
$\Delta_{\shuffle}(h)=\sum h_1^{(1)}h_2\otimes h_1^{(2)}+\sum h_1^{(1)} \otimes h_1^{(2)}h_2$, so the result holds for $h$. \end{proof}

\subsection{Associated Hopf algebra and post-Lie algebra}

\begin{theorem} \label{theo4}
Let $*$ be a magmatic product on $V$. This product is extended to $T(V)$ by Proposition \ref{prop2}.
We define a product $\circledast$ on $T(V)$ by:
$$\forall f,g\in T(V),\:f\circledast g=\sum \left(f*g^{(1)}\right) g^{(2)}.$$
Then $(T(V),\circledast,\Delta_{\shuffle})$ is a Hopf algebra.
\end{theorem}

\begin{proof} For all $f\in T(V)$:
\begin{align*}
1 \circledast f&\sum \left(1*f^{(1)}\right) f^{(2)}=\sum \varepsilon \left(f^{(1)}\right)f^{(2)}=f;&\circledast 1=(f*1)1=f.
\end{align*}
For all $f,g,h \in T(V)$, by Lemma \ref{lem3}-5:
\begin{align*}
(f \circledast g)\circledast h&=\sum \left(\left(f*g^{(1)}\right)g^{(2)}\right)\circledast h\\
&=\sum\left(\left(\left(f*g^{(1)}\right)g^{(2)}\right)*h^{(1)}\right)h^{(2)}\\
&=\sum \left(\left(f*g^{(1)}\right)*h^{(1)}\right)\left(g^{(2)}*h^{(2)}\right) h^{(3)}\\
&=\sum \left(f*\left(\left(g^{(1)}*h^{(1)}\right)h^{(2)}\right)\right)\left(g^{(2)}*h^{(3)}\right) h^{(4)};\\
f \circledast (g \circledast h)&=\sum f\circledast \left(\left(g*h^{(1)}\right) h^{(2)}\right)\\
&=\sum \left(f*\left(\left(g^{(1)}*h^{(1)}\right)h^{(3)}\right)\right)\left(g^{(2)}*h^{(2)}\right) h^{(4)}.
\end{align*}
As $\Delta_{\shuffle}$ is cocommutative, $(f \circledast g)\circledast h=f\circledast (g\circledast h)$, so $(T(V),\circledast)$ is a unitary, associative algebra.\\

For all $f,g \in T(V)$, by Lemma \ref{lem3}-3:
\begin{align*}
\Delta_{\shuffle}(f \circledast g)&=\sum \Delta_{\shuffle}\left(\left(f*g^{(1)}\right)g^{(2)}\right)\\
&=\sum \left(f^{(1)}*\left(g^{(1)}\right)^{(1)}\right) \left(g^{(2)}\right)^{(1)}\otimes \left(f^{(2)}*\left(g^{(1)}\right)^{(2)}\right)\left(g^{(2)}\right)^{(2)}\\
&=\sum \left(f^{(1)}*\left(g^{(1)}\right)^{(1)}\right)\left(g^{(1)}\right)^{(2)} \otimes \left(f^{(2)}*\left(g^{(2)}\right)^{(1)}\right)\left(g^{(2)}\right)^{(2)}\\
&=\sum f^{(1)}\circledast g^{(1)} \otimes f^{(2)}\circledast g^{(2)}.
\end{align*}
Note that we used the cocommutativity of $\Delta_{\shuffle}$ for the third equality. Hence, $(T(V), \circledast,\Delta_{\shuffle})$ is a Hopf algebra. 
\end{proof}\\

\textbf{Remark.} By Lemma \ref{lem3}:
\begin{itemize}
\item For all $f,g,h\in T(V)$, $(f*g)*h=f*(g\circledast h)$: $(T(V),*)$ is a right $(T(V),\circledast)$-module.
\item By restriction, for all $n\geq 0$, $(V^{\otimes n},*)$ is a right $(T(V),\circledast)$-module.
Moreover, for all $n\geq 0$, $(V^{\otimes n},*)=(V,*)^{\otimes n}$ as a right module over the Hopf algebra $(T(V), \circledast,\Delta_{\shuffle})$.
\end{itemize}

 \textbf{Examples.} Let $x_1,x_2,x_3 \in V$.
 \begin{align*}
 x_1\circledast x_2&=x_1x_2+x_1*x_2\\
 x_1\circledast x_2x_3&=x_1x_2x_3+(x_1*x_2)x_3+(x_1*x_3)x_2+(x_1*x_2)*x_3-x_1*(x_2*x_3)\\
 x_1x_2\circledast x_3&=x_1x_2x_3+(x_1*x_3)x_2+x_1(x_2*x_3).
 \end{align*}
 
The vector space of primitive elements of $(T(V), \circledast,\Delta_{\shuffle})$ is $\Lie(V)$.
Let us now describe the Lie bracket induced on $\Lie(V)$ by $\circledast$.
 
\begin{proposition} \label{prop5}
\begin{enumerate}
\item Let $*$ be a magmatic product on $V$. The Hopf algebras $(T(V),\circledast,\Delta_{\shuffle})$ and $(T(V),.,\Delta_{\shuffle})$ are isomorphic,
via the following algebra morphism:
$$\phi_*:\left\{\begin{array}{rcl}
(T(V),.,\Delta_{\shuffle})&\longrightarrow&(T(V),\circledast,\Delta_{\shuffle})\\
x_1\ldots x_k\in V^{\otimes k}&\longrightarrow&x_1\circledast \ldots \circledast x_k.
\end{array}\right.$$
\item $\Lie(V)*T(V)\subseteq \Lie(V)$. Moreover, $(\Lie(V),[-,-],*)$ is a post-Lie algebra.
The induced Lie bracket on $\Lie(V)$ is denoted by $\{-,-\}_*$:
$$\forall f,g\in \Lie(V),\: \{f,g\}_*=[f,g]+f*g-g*f=fg-gf+f*g-g*f.$$
The Lie algebra $(\Lie(V),\{-,-\}_*)$ is isomorphic to $\Lie(V)$.
\end{enumerate}\end{proposition} 

\begin{proof} 1. There exists a unique algebra morphism $\phi_*:(T(V),.)\longrightarrow (T(V),\circledast)$,
sending any $x\in V$ on itself. As the elements of $V$ are primitive in both Hopf algebras, $\phi_*$ is a Hopf algebra morphism.
As $V^{\otimes k}*T(V)\subseteq V^{\otimes k}$ for all $k\geq 0$, we deduce that for all $x_1,\ldots,x_{k+l} \in V$:
$$x_1\ldots x_k\circledast x_{k+1}\ldots x_{k+l}=x_1\ldots x_{k+l}+\mbox{a sum of words of length $<k+l$}.$$
Hence, if $x_1,\ldots,x_k \in V$:
$$\phi_*(x_1\ldots x_k)=x_1\circledast\ldots \circledast x_k=x_1\ldots x_k+\mbox{a sum of words of length $<k$}.$$
Consequently:
\begin{itemize}
\item If $k\geq 0$ and $x_1,\ldots,x_k\in V$, an induction on $k$ proves that $x_1\ldots x_k\in \phi_*(T(V))$, so $\phi_*$ is surjective.
\item If $f$ is a nonzero element of $T(V)$, let us write $f=f_0+\ldots+f_k$, with $f_i\in V^{\otimes i}$ for all $i$ and $f_k\neq 0$.
Then:
$$\phi_*(f)=f_k+\mbox{terms in }\K\oplus \ldots \oplus V^{\otimes (k-1)},$$
so $\phi_*(f)\neq 0$: $\phi_*$ is injective.
\end{itemize}
Hence, $\phi_*$ is an isomorphism. \\

2.  We consider $A=\{f\in \Lie(V)\mid f*T(V)\subseteq \Lie(V)\}$.
By Lemma \ref{lem3}-3, $V\subseteq A$. Let $f,g \in A$. For all $h\in T(V)$:
\begin{align*}
[f,g]*h&=(fg)*h-(gf)*h\\
&=\sum \left(f*h^{(1)}\right)\left(g*h^{(2)}\right)-\sum\left(g*h^{(1)}\right)\left(f*h^{(2)}\right)\\
&=\sum \left(f*h^{(1)}\right)\left(g*h^{(2)}\right)-\sum\left(g*h^{(2)})(f*h^{(1)}\right)\\
&=\sum\left[f*h^{(1)},g*h^{(2)}\right].
\end{align*}
We used the cocommutativity for the third equality. By hypothesis, $f*h^{(1)}$, $g*h^{(2)} \in \Lie(V)$,
so $[f,g]\in A$. As $A$ is a Lie subalgebra  of $\Lie(V)$ containing $V$, it is equal to $\Lie(V)$.\\

Let $f,g,h \in \Lie(V)$. Then $g \circledast h=\sum\left(g*h^{(1)}\right)h^{(2)}=gh+g*h$. 
Similarly, $\sum\left(h*g^{(1)}\right)g^{(2)}=hg+h*g$, so, by Lemma \ref{lem3}-5:
\begin{align*}
f*[g,h]&=f*(gh)-f*(hg)\\
&=\sum f*\left(\left(g*h^{(1)}\right)h^{(2)}\right)-f*(g*h)-\sum f*\left(\left(h*g^{(1)}\right)g^{(2)}\right)+f*(h*g)\\
&=(f*g)*h-f*(g*h)-(f*h)*g+f*(g*h).
\end{align*}
Moreover:
\begin{align*}
[f,g]*h&=(fg)*h-(gf)*h\\
&=(f*h)g+f(g*h)-(g*h)f-g(f*h)\\
&=[f*h,g]+[f,g*h].
\end{align*}
So $\Lie(V)$ is a post-Lie algebra. \\

Consequently, $\{-,-\}_*$ is a second Lie bracket on $\Lie(V)$. In $(T(V),\circledast)$, if $f$ and $g$ are primitive:
$$f\circledast g-g\circledast f=fg+f*g-gf-g*f=\{f,g\}_*.$$
So, by the Cartier-Quillen-Milnor-Moore's theorem, $(T(V),\circledast,\Delta_{\shuffle})$ is the enveloping algebra of $(\Lie(V),\{-,-\}_*)$.
As it is isomorphic to the enveloping algebra of $\Lie(V)$, namely $(T(V),.,\Delta_{\shuffle})$, these two Lie algebras are isomorphic. \end{proof}\\

Let us give a combinatorial description of $\phi_*$.

\begin{proposition} Let $(V,*)$ be a magmatic algebra, and $x_1,\ldots,x_k\in V$.
\begin{itemize}
\item Let $I=\{i_1,\ldots,i_p\}\subseteq [k]$, with $i_1<\ldots<i_p$. We put:
$$x_I^*=(\ldots ((x_{i_1}*x_{i_2})*x_{i_3})*\ldots)*x_{i_p} \in V.$$
\item Let $P$ be a partition of $[p]$. We denote it by $P=\{P_1,\ldots,P_p\}$, with the convention $\min(P_1)<\ldots<\min(P_p)$.
We put:
$$x_P^*=x_{P_1}^* \ldots x_{P_p}^*\in V^{\otimes p}.$$
\end{itemize} 
Then:
$$\phi^*(x_1\ldots x_k)=\sum_{\mbox{\scriptsize $P$ partition of $[k]$}} x_P^*.$$
\end{proposition}

\begin{proof} By induction on $k$. As $\phi_*(x)=x$ for all $x\in V$, it is obvious if $k=1$. Let us assume the result at rank $k$.
\begin{align*}
\phi_*(x_1\ldots x_{k+1})&=\phi_*(x_1\ldots x_k) \circledast x_{k+1}\\
&=\phi_*(x_1\ldots x_k)x_{k+1}+\phi_*(x_1\ldots x_k)*x_{k+1}\\
&=\sum_{\mbox{\scriptsize $P$ partition of $[k]$}}x_P^*x_{k+1}+
\sum_{\substack{\mbox{\scriptsize $P=\{P_1,\ldots,P_p\}$}\\ \mbox{\scriptsize partition of $[k]$}}}
\sum_{i=1}^p x_{P_1}^*\ldots (x_{P_i}^**x_{k+1})\ldots x_{p_p}^*\\
&=\sum_{\substack{\mbox{\scriptsize $P=\{P_1,\ldots,P_p\}$}\\ \mbox{\scriptsize partition of $[k]$}}}
x_{\{P_1,\ldots,P_p,\{k+1\}\}}^*+\sum_{\substack{\mbox{\scriptsize $P=\{P_1,\ldots,P_p\}$}\\ \mbox{\scriptsize partition of $[k]$}}}
\sum_{i=1}^p x_{\{P_1,\ldots,P_i\cup\{k+1\},\ldots,P_p\}}^*\\
&=\sum_{\mbox{\scriptsize $P$ partition of $[k+1]$}}x_P^*.
\end{align*}
So the result holds for all $k$. \end{proof}\\

\textbf{Examples.} Let $x_1,x_2,x_3\in V$.
\begin{align*}
\phi_*(x_1)&=x_1,\\
\phi_*(x_1x_2)&=x_1x_2+x_1*x_2,\\
\phi_*(x_1x_2x_3)&=x_1x_2x_3+(x_1*x_2)x_3+(x_1*x_3)x_2+x_1(x_2*x_3)+(x_1*x_2)*x_3.
\end{align*}

\begin{theorem}\label{theo7}
Let $(V,*)$ be a magmatic algebra and let $(L,\{-,-\},\star)$ be a post-Lie algebra. Let $\phi:(V,*)\longrightarrow (L,\star)$ be a morphism
of magmatic algebras. There exists a unique morphism of post-Lie algebras $\overline{\phi} : \Lie(V)\longrightarrow L$ extending $\phi$.
\end{theorem}

\begin{proof} Let $\psi:\Lie(V)\longrightarrow L$ be the unique Lie algebra morphism extending $\phi$. 
Let us fix $h\in \Lie(V)$. We consider:
$$A_h=\{h\in \Lie(V)\:\mid\: \forall f \in \Lie(V), \: \psi(f*h)=\psi(f) \star\psi(h)\}.$$
If $f,g \in A_h$, then:
\begin{align*}
\psi([f,g]*h)&=\psi([f*h,b]+[f,g*h])\\
&=\{\psi(f*h),\psi(g)\}+\{\psi(f),\psi(g*h)\}\\
&=\{\psi(f)\star\psi(h),\psi(g)\}+\{\psi(f),\psi(g)\star\psi(h)\}\\
&=\{\psi(f),\psi(g)\}\star\psi(h)\\
&=\psi([f,g])\star\psi(h).
\end{align*}
So $[f,g]\in A_h$: for all $h\in \Lie(V)$, $A_h$ is a Lie subalgebra of $\Lie(V)$.
Moreover, if $h\in V$, as $\psi_{\mid V}=\phi$ is a morphism of magmatic algebras, $V \subseteq A_h$; as a consequence, if $h\in V$, $A_h=\Lie(V)$. \\

Let $A=\{h\in \Lie(V)\:\mid\: A_h=\Lie(V)\}$. We put $\Lie(V)_n=\Lie(V) \cap V^{\otimes n}$;
let us prove  inductively that $\Lie(V)_n \subseteq A$ for all $n$. We already proved that $V \subseteq A$, so this is true for $n=1$.
Let us assume the result at all rank $k<n$. Let $h\in \Lie(V)_n$. We can assume that $h=[h_1,h_2]$, with $h_1 \in \Lie(V)_k$, $h_2\in \Lie(V)_{n-k}$,
$1\leq k\leq n-1$. From Lemma \ref{lem3} and Proposition \ref{prop5}, $_1f*h_2 \in \Lie(V)_k$ and $h_2*h_1\in \Lie(V)_{n-k}$, 
so the induction hypothesis holds for $h_1$, $h_2$, $h_1*h_2$ and $h_2*h_1$. Hence, for all $f\in T(V)$:
\begin{align*}
\psi(f*h)&=\psi(f*[h_1,h_2])\\
&=\psi((f*h_1)*h_2-f*(h_1*h_2)-(f*h_2)*h_1+f*(h_2*h_1))\\
%&=\psi(f*h_1)\star\psi(h_2)-\psi(f)\star\psi(h_1*h_2)-\psi(f*h_2)\star\psi(h_1)+\psi(f)\star\psi(h_2*h_1)\\
&=(\psi(f)\star\psi(h_1))\star\psi(h_2)-\psi(f)\star(\psi(h_1)\star\psi(h_2))\\
&-(\psi(f)\star\psi(h_2))\star\psi(h_1)+\psi(f)\star(\psi(h_2)\star\psi(h_1))\\
&=\psi(f) \star\{\psi(h_1),\psi(h_2)\}\\
&=\psi(f) \star\psi(h).
\end{align*}
As a consequence, $\Lie(V)_n \subseteq A$. Finally, $A=\Lie(V)$, so for all $f,g\in \Lie(V)$, $\psi(f*g)=\psi(f)*\psi(g)$. \end{proof}
 
\begin{corollary}\label{cor8}
Let $V$ be a vector space. The free magmatic algebra generated by $V$ is denoted by $\Mag(V)$.
Then $\Lie(\Mag(V))$ is the free post-Lie algebra generated by $V$.
\end{corollary}

\begin{proof} Let $L$ be a post-Lie algebra and let $\phi$ be a linear map from $V$ to $L$. From the universal property of $\Mag(V)$, there exists a unique
morphism of magmatic algebras from $\Mag(V)$ to $L$ extending $\phi$; from the universal property of $\Lie(\Mag(V))$, this morphism can be uniquely
extended as a morphism of post-Lie algebras from $\Lie(\Mag(V))$ to $V$. So $\Lie(\Mag(V))$ satisfies the required universal property to 
be a  post-Lie algebra generated by $V$. \end{proof} \\

\textbf{Remark.} Describing the free magmatic algebra generated by $V$ is terms of planar rooted trees with a grafting operation,
we get back the construction of free post-Lie algebras of \cite{MK}.

\subsection{Enveloping algebra of a post-Lie algebra}

Let $(V,\{-,-\},*)$ be a post-Lie algebra.  We extend $*$ onto $T(V)$ as previously in Proposition \ref{prop2}. 
The usual bracket of $\Lie(V) \subseteq T(V)$ is denoted by $[f,g]=fg-gf$, and should not be confused with the bracket $\{-,-\}$ of the post-Lie algebra $V$.

\begin{lemma}
Let $I$ be the two-sided ideal of $T(V)$ generated by the elements $xy-yx-\{x,y\}$, $x,y \in V$.
Then $I*T(V)\subseteq I$ and $T(V)*I=(0)$.
\end{lemma}

\begin{proof} \textit{First step.} Let us prove that for all $x,y\in V$, for all $h \in T(V)$:
$$\{x,y\}*h=\sum \left\{x*h^{(1)}, y*h^{(2)}\right\}.$$
Note that the second member of this formula makes sense, as $V*T(V) \subseteq V$ by Lemma \ref{lem3}.\\
We assume that $h=z_1\ldots z_n$ and we work by induction on $n$. If $n=0$, then $h=1$ and $\{x,y\}*1=\{x,y\}=\{x*1,y*1\}$.
If $n=1$, then $h\in V$, so $\Delta_{\shuffle}(h)=h\otimes 1+1\otimes h$.
$$\{x,y\}*h=\{x*h,y\}+\{x,y*h\}=\{x*h,y*1\}+\{x*1,y*h\}=\sum \{x*h^{(1)},y*h^{(2)}\}.$$
If $n \geq 2$, we put $h_1=z_1\ldots z_{n-1}$ and $h_2=z_n$. The induction hypothesis holds for $h_1$, $h_2$ and $h_1*h_2$:
\begin{align*}
\{x,y\}*h&=(\{x,y\}*h_1)*h_2-\{x,y\}*(h_1*h_2)\\
&=\sum \left\{x*h_1^{(1)}, y*h_1^{(2)}\right\}*h_2-\sum \left\{x*\left(h_1*h_2\right)^{(1)} ,y*\left(h_1*h_2\right)^{(2)}\right\}\\
&=\sum \left\{\left(x*h_1^{(1)}\right)*h_2^{(1)}, \left(y*h_1^{(2)}\right)*h_2^{(2)}\right\}
-\sum \left\{x*\left(h_1^{(1)}*h_2^{(1)}\right) ,y*\left(h_1 ^{(2)}*h_2^{(2)}\right)\right\}\\
&=\sum \left\{\left(x*h_1^{(1)}\right)*h_2, y*h_1^{(2)}\right\}+\sum \left\{x*h_1^{(1)}, \left(y*h_1^{(2)}\right)*h_2\right\}\\
&-\sum \left\{x*\left(h_1^{(1)}*h_2\right) ,y*h_1 ^{(2)}\right\}-\sum \left\{x*h_1^{(1)},y*\left(h_1 ^{(2)}*h_2\right)\right\}\\
&=\sum \left\{\left(x*h_1^{(1)}\right)*h_2-x*\left(h_1^{(1)}*h_2\right),y*h_1^{(2)}\right\}\\
&+\sum \left\{x*h_1^{(1)}, \left(y*h_1^{(2)}\right)*h_2-y*\left(h_1^{(2)}*h_2\right)\right\}\\
&=\sum \left\{x*\left(h_1^{(1)}h_2\right),y*h_1^{(2)}\right\}+\sum \left\{x*h_1^{(1)}, y*\left(h_1^{(2)}h_2\right)\right\}\\
&=\sum \left\{x*h^{(1)},y*h^{(2)}\right\}.
\end{align*}
Consequently, the result holds for all $h\in T(V)$.\\

\textit{Second step.} Let $J=Vect(xy-yx-\{x,y\}\:\mid \:x,y\in V)$. For all $x,y \in V$, for all $h\in T(V)$, by the first step:
$$(xy-yx-\{x,y\})*h=\sum\left(x*h^{(1)}\right)\left(y*h^{(2)}\right)-\left(y*h^{(1)}\right)\left(y*h^{(2)}\right)-\left\{x*h^{(1)},y*h^{(2)}\right\}\in J.$$
So $J*T(V) \subseteq J$. If $g\in J$, $f_1,f_2,h \in T(V)$:
$$(f_1gf_2)*h=\sum \left(f_1*d^{(1)}\right)\underbrace{\left(g*h^{(2)}\right)}_{\in J} \left(f_2*h^{(3)}\right) \in I.$$
So $I*T(V) \subseteq I$. \\

Let us prove that $T(V)*(T(V)JV^{\otimes n})=(0)$ for all $n \geq 0$. We start with $n=0$. 
First, $1*(T(V)J)=\varepsilon(T(V)J)=(0)$. Let $x,y,z \in V$, $g\in T(V)$. Then:
\begin{align*}
&x*(gyz-gzy-g\{y,z\})\\
&=(x*(gy))*z-x*((gy)*z)-(x*(gz))*y+x*((gz)*y)\\
&-(x*g)*\{y,z\}+x*(g*\{y,z\})\\
&=((x*g)*y)*z-(x*(g*y)*z-x*((g*z)y)\\
&-x*(g(y*z))-((x*g)*z)*y-(x*(g*z))*y\\
&+x*((g*y)z)+x*(g(z*y))-(x*g)*\{y,z\}+x*(g*\{y,z\})\\
&=((x*g)*y)*z-(x*(g*y))*z-(x*(g*z))*y+x*((g*z)*y)\\
&-(x*g)*(y*z)+x*(g*(y*z))-((x*g)*z)*y+(x*(g*z))*y\\
&(x*(g*y))*z-x*((g*y)*z)+(x*g)*(z*y)-x*(g*(z*y))\\
&-(x*g)*\{y,z\}+x*(g*\{y,z\})\\
&=x*((g*z)*y)+x*(g*(y*z))-x*((g*y)*z)-x*(g*(z*y))+x*(g*\{y,z\})\\
&+((x*g)*y)*z-(x*g)*(y*z)-((x*g)*z)*y+(x*g)*(z*y)-(x*g)*\{y,z\}\\
&=0+0.
\end{align*}
So $V*(T(V)J)=(0)$. As the elements of $J$ are primitive, $T(V)J$ is a coideal. If $n \geq 1$,
 $x_1,\ldots,x_n\in V$ and $g\in T(V)J$, we put $\Delta_{\shuffle}^{(n-1)}(g)=\sum g^{(1)}\otimes\ldots \otimes g^{(n)}$,
with at least one $g_i \in T(V)J$. Then $(x_1\ldots x_n)*g=\sum (x_1*g^{(1)})\ldots (x_n*g^{(n)})=0$,
so $T(V)*(T(V)J)=(0)$.

If $n\geq 1$, we take $f \in T(V)$, $g\in T(V)JV^{\otimes (n-1)}$ and $y\in V$.  
We put $g=g_1g_2g_3$, with $g_1\in T(V)$, $g_2 \in J$, $g_3\in V^{\otimes (n-1)}$. Then:
$$g*y=(g_1*y) g_2 g_3+g_1 \underbrace{(g_2*y)}_{\in J*T(V) \subseteq J} g_3
+g_1 g_2 \underbrace{(g_3*y)}_{\in V^{\otimes n}} \in T(V)JV^{\otimes n}.$$
So the induction hypothesis holds for $g$ and for $g*y$. Then $f*(gy)=(f*g)*y-f*(g*y)=0$.
So $T(V)*I=(0)$. \end{proof}\\

As a consequence, the quotient $T(V)/I$ inherits  a magmatic product $*$. Moreover, $I$ is a Hopf ideal, and this implies that
 it is also a two-sided ideal for $\circledast$. As $T(V)/I$  is the enveloping algebra $\U(V,\{-,-\})$, we obtain Proposition 3.1 of \cite{Ebrahimi}:

\begin{proposition}
Let $(\g,\{-,-\},*)$ be a post-Lie algebra. Its magmatic product can be uniquely extended to $\U(\g)$ such that for all $f,g,h\in \U(\g)$, $x,y\in \g$:
\begin{itemize}
\item $f*1=f$.
\item $1*f=\varepsilon(f)1$.
\item $f*(gy)=(f*g)*y-f*(g*y)$.
\item $(fg)*h=\sum \left(f*h^{(1)}\right)\left(g*h^{(2)}\right)$, where $\Delta(h)=\sum h^{(1)}\otimes h^{(2)}$ is the usual coproduct of $\U(\g)$.
\end{itemize}
We define a product $\circledast$ on $\U(\g)$ by $f*g=\sum \left(f*g^{(1)}\right) g^{(2)}$. Then $(\U(\g),\circledast,\Delta)$ is a Hopf algebra,
isomorphic to $\U(\g,\{-,-\}_*)$.
\end{proposition}

\begin{proof} By Cartier-Quillen-Milnor-Moore's theorem, $(\U(\g),\circledast,\Delta)$ is an enveloping algebra; the underlying Lie algebra
is $Prim(\U(\g))=\g$, with the Lie bracket defined by:
\begin{align*}
\{x,y\}_\circledast&=x\circledast y-y\circledast x=xy+x*y-yx-y*x.
\end{align*}
This is the bracket $\{-,-\}_*$. \end{proof}\\

\textbf{Remarks.} \begin{enumerate}
\item If $\g$ is a post-Lie algebra with $\{-,-\}=0$, it is a pre-Lie algebra, and $\U(\g)=S(\g)$. We obtain again the Oudom-Guin construction
\cite{Oudom2,Oudom1}.
\item By Lemma \ref{lem3}, $(\U(\g),*)$ is a right $(\U(\g),\circledast)$-module. By restriction, $(\g,*)$ is also a right $(\U(\g),\circledast)$-module.
\end{enumerate}

\subsection{The particular case of associative algebras}

Let $(V,\triangleleft)$ be an associative algebra. The associated Lie bracket is denoted by $[-,-]_\triangleleft$. As $(V,0,\triangleleft)$
is post-Lie, the construction of the enveloping algebra of $(V,[-,-]_\triangleleft)$ can be done: we obtain a product $\triangleleft$
defined on $S(V)$ and an associative product $\blacktriangleleft$ making $(S(V),\blacktriangleleft,\Delta)$ a Hopf algebra,
isomorphic to the enveloping algebra of $(V,[-,-]_\triangleleft)$. 

\begin{lemma}\label{lem12}
If $x_1,\ldots,x_k,y_1,\ldots,y_l \in V$:
\begin{align*}
x_1\ldots x_k \triangleleft y_1\ldots y_l&=\sum_{\theta:[l]\hookrightarrow [k]} \left(\prod_{i\notin Im(\theta)} x_i\right)
\left(\prod_{i=1}^k x_{\theta(i)} \triangleleft y_i\right),\\
x_1\ldots x_k \blacktriangleleft y_1\ldots y_l&=\sum_{I\subseteq [l]}\sum_{\theta:I\hookrightarrow [k]} 
\left(\prod_{i\notin Im(\theta)} x_i\right)\left(\prod_{j\notin I} y_j\right)\left(\prod_{i\in I}x_{\theta(i)} \triangleleft y_i\right).
\end{align*}\end{lemma}

\begin{proof}  We first prove that for all $k\geq 2$, $x,y_1,\ldots,y_k \in V$, $x\triangleleft y_1\ldots y_k=0$.
We proceed by induction on $k$. For $k=2$, $x\triangleleft y_1y_2=(x\triangleleft y_1)\triangleleft y_2-x\triangleleft (y_1\triangleleft y_2)=0$,
as $\triangleleft$ is associative. Let us assume the result at rank $k$. Then:
$$x\triangleleft y_1\ldots y_{k+1}=(x\triangleleft y_1\ldots y_k)\triangleleft y_{k+1}-\sum_{i=1}^k 
x\triangleleft (y_1\ldots (y_i\triangleleft y_{k+1})\ldots y_k)=0.$$

Let us now prove the formula for $\triangleleft$. 
\begin{align*}
x_1\ldots x_k \triangleleft y_1\ldots y_l&=\sum_{[l]=I_1\sqcup \ldots \sqcup I_k}
\left(x_1\triangleleft \prod_{i\in I_1} y_i\right)\ldots \left(x_k\triangleleft \prod_{i\in I_k} y_i\right).
\end{align*}
Moreover, for all $j$:
$$x_j\triangleleft \prod_{i\in I_j} y_i=\begin{cases}
x_j\mbox{ if } I_j=\emptyset,\\
x_j\triangleleft y_p\mbox{ if }I_j=\{p\},\\
0\mbox{ otherwise}.
\end{cases}$$
Hence:
\begin{align*}
x_1\ldots x_k \triangleleft y_1\ldots y_l&=\sum_{\substack{[l]=I_1\sqcup \ldots \sqcup I_k\\ \forall p, \: |I_p|\leq 1}}
\left(x_1\triangleleft \prod_{i\in I_1} y_i\right)\ldots \left(x_k\triangleleft \prod_{i\in I_k} y_i\right)\\
&=\sum_{\theta:[l]\hookrightarrow [k]} \left(\prod_{i\notin Im(\theta)} x_i\right)
\left(\prod_{i=1}^k x_{\theta(i)} \triangleleft y_i\right).
\end{align*}
Finally:
$$x_1\ldots x_k\blacktriangleleft y_1\ldots y_l=\sum_{I\subseteq [l]} 
\left(\prod_{i\notin I} y_i\right) x_1\ldots x_k \triangleleft \left(\prod_{i\in I} y_i\right),$$
as announced. \end{proof}\\

\textbf{Examples.} Let $x_1,x_2,y_2,y_2 \in V$.
\begin{align*}
x_1\blacktriangleleft y_1&=x_1y_1+x_1\triangleleft y_1,\\
x_1x_2\blacktriangleleft y_1&=x_1x_2y_1+(x_1\triangleleft y_1)x_2+x_1(x_2\triangleleft y_1),\\
x_1 \blacktriangleleft y_1y_2&=x_1y_1y_2+(x_1\triangleleft y_1)y_2+(x_1\triangleleft y_2)y_1,\\
x_1x_2\blacktriangleleft y_1y_2&=x_1x_2y_1y_2+(x_1\triangleleft y_1)x_2y_2+(x_1\triangleleft y_2)x_2y_1+x_1(x_2\triangleleft y_1)y_2\\
&+x_1(x_2\triangleleft y_2)y_1+(x_1\triangleleft y_1)(x_2\triangleleft y_2)+(x_1\triangleleft y_2)(x_2\triangleleft y_1).
\end{align*}

\textbf{Remark.} The number of terms in $x_1\ldots x_k \triangleleft y_1\ldots y_l$ is:
$$\sum_{i=0}^{\min(k,l)} \binom{l}{i}\binom{k}{i}i!,$$
see sequences A086885 and A176120 of \cite{Sloane}.

\section{A family of solvable Lie algebras}

\subsection{Definition}

\begin{definition} 
 Let us fix $a=(a_1,\ldots,a_N)\in \K^N$. We define an associative product $\triangleleft$ on $\K^N$:
 $$\forall i,j\in [N],\: \epsilon_i\triangleleft \epsilon_j=a_j\epsilon_i.$$
The associated Lie bracket is denoted by $[-,-]_a$:
$$\forall i,j\in[N],\: [\epsilon_i,\epsilon_j]_a=a_j\epsilon_i-a_i\epsilon_j.$$
This Lie algebra is denoted by $\g_a$.
\end{definition}

%\begin{proposition} \label{prop14}
%I
%\end{proposition}
%
%\begin{proof} Let us assume that $a\neq (0,\ldots,0)$. Let us consider the linear form $d$ on $\K^n$ which sends $\epsilon_i$ on $a_i$ for all $i\in [N]$.
%This is a non zero linear form, and for all $x,y\in \K^n$, $[x,y]_a=f(y)x-f(x)y$.
%Let $(\epsilon'_1,\ldots,\epsilon'_n)$ be a basis of $\K^n$ such that $d(\epsilon'_i)=\delta_{i,1}$ for all $i\in [N]$.
%We put:
%$$\theta:\left\{\begin{array}{rcl}
%\g_{(1,0,\ldots,0)}&\longrightarrow&\g_a\\
%\epsilon_i&\longrightarrow&\epsilon'_i.
%\end{array}\right.$$
%Then $\theta$  is a Lie algebra isomorphism. \end{proof}

\textbf{Remarks.} \begin{enumerate}
\item Let $A\in M_{N,M}(\K)$, and $a\in \K^N$. The following map is a Lie algebra morphism:
$$\left\{\begin{array}{rcl}
\g_{a. \:^tA}&\longrightarrow&\g_a\\
x&\longrightarrow&Ax.
\end{array}\right.$$
Consequently, if $a\neq(0,\ldots,0)$, $g_a$ is isomorphic to $\g_{(1,0,\ldots,0)}$.
\item These Lie algebras $\g_a$ are characterized by the following property: if $\g$
 is a $n$-dimensional Lie algebra such that any $2$-dimensional subspace of $\g$ is a Lie subalgebra, there exists $a\in \K^n$
 such that $\g$ and $\g_a$ are isomorphic.
\end{enumerate}

\begin{definition}
Let $A=T(V)^N$. The elements of $A$ will be denoted by:
$$f=\left(\begin{array}{c} f_1\\ \vdots\\ f_N\end{array}\right)=f_1\epsilon_1+\ldots+f_N\epsilon_N.$$
For all $i,j \in [N]$, we define bilinear products $\shufg{i}$ and $\shufd{j}$:
\begin{align*}
&\forall  f,g \in T(V)^N,&f \shufg{i} g&=\left(\begin{array}{c} f_i \shuffle g_1\\ \vdots\\ f_i \shuffle g_N\end{array}\right),&
f \shufd{j} g&=\left(\begin{array}{c} f_1 \shuffle g_j\\ \vdots\\ f_N \shuffle g_j\end{array}\right).
\end{align*}
In other words, if $f,g \in T(V)$, for all $k,l\in [N]$:
\begin{align*}
f\epsilon_k \shufg{i} g\epsilon_l&=\delta_{i,k} (f\shuffle g) \epsilon_l,&
f\epsilon_k \shufd{j} g\epsilon_l&=\delta_{j,l} (f\shuffle g) \epsilon_k.
\end{align*}
If $a=(a_1,\ldots,a_N)\in \K^N$, we put $_a\shuffle=a_1\shufg{1}+\ldots+a_N\shufg{N}$ and $\shufd{a}=a_1\shufd{1}+\ldots+a_N \shufd{N}$.
\end{definition}

\begin{proposition}
Let $f,g \in \K^N$. For all $f,g,h\in A$:
\begin{align*}
(f\shufd{a} g)\shufd{b} h&=f \shufd{a} (g\shufd{b} h),&(f\shufd{a} g)\shufg{b} h&=f \shufd{a} (g\shufg{b} h),\\
(f\shufg{a} g)\shufd{b} h&=f \shufg{a} (g\shufd{b} h),&(f\shufg{a} g)\shufg{b} h&=f \shufg{a} (g\shufg{b} h),\\
f \shufd{a} g&=g\shufg{a} f.
\end{align*}\end{proposition}

\begin{proof} Direct verifications, using the associativity and the commutativity of $\shuffle$. \end{proof}

\begin{definition}
Let $a\in \K^N$. We define a Lie bracket on $A$:
$$\forall f,g \in A,\: \:_a[f,g]=f \shufg{a} g-g \shufg{a} f=g\shufd{a} f-f\shufd{a} g.$$
This Lie algebra is denoted by $\g'_a$.
\end{definition}

\textbf{Remark.} If $A$ is an associative commutative algebra and $\g$ is a Lie algebra, then $A\otimes \g$ is a Lie algebra, with the following Lie bracket:
$$\forall f,g \in A,\: x,y\in \g,\: [f\otimes x,g\otimes y]=fg\otimes [x,y].$$
Then, as a Lie algebra, $\g'_a$ is isomorphic to the tensor product of the associative commutative algebra $(T(V),\shuffle)$, and of the Lie algebra $\g_{-a}$.
Consequently, if $a\neq (0,\ldots,0)$, $\g'_a$ is isomorphic to $\g'_{(1,0,\ldots,0)}$.

%\begin{proof} The isomorphism is given by:
%$$\left\{\begin{array}{rcl}
%T(V)\otimes \g_{-a}&\longrightarrow&A\\
%f\otimes \epsilon_i&\longrightarrow&f\epsilon_i.
%\end{array}\right.$$
%If $a\neq (0,\ldots,0)$, $\g_{-a}$  and $\g_{-(1,0,\ldots,0)}$ are isomorphic by Proposition \ref{prop14}, so $\g'_a$ and $\g'_{(1,0,\ldots,0)}$ are isomorphic.
%\end{proof}

\subsection{Enveloping algebra of $\g_a$}

Let us apply Lemma \ref{lem12} to the Lie algebra $\g_a$:

\begin{proposition}\label{prop19}
The symmetric algebra $S(\g_a)$ is given an associative product $\blacktriangleleft$ such that for all $i_1,\ldots,i_k,j_1,\ldots,j_l \in [N]$:
$$\epsilon_{i_1}\ldots \epsilon_{i_k}\blacktriangleleft \epsilon_{j_1}\ldots \epsilon_{j_l}
=\sum_{I\subseteq [l]} k(k-1)\ldots (k-|I|+1) \left(\prod_{q\in I} a_{j_q}\right)\left(\prod_{p\notin I} \epsilon_{j_p}\right)\epsilon_{i_1}\ldots \epsilon_{i_k}.$$
The Hopf algebra $(S(\g_a),\blacktriangleleft,\Delta)$ is isomorphic to the enveloping algebra of $\g_a$.
\end{proposition}

The enveloping algebra of $\g_a$ has two distinguished bases, the Poincaré-Birkhoff-Witt basis and the monomial basis:
\begin{align*}
&(\epsilon_{i_1}\blacktriangleleft\ldots \blacktriangleleft \epsilon_{i_k})_{k\geq 0,\: 1\leq i_1\leq \ldots \leq i_k\leq N},
&&(\epsilon_{i_1}\ldots \epsilon_{i_k})_{k\geq 0,\: 1\leq i_1\leq \ldots \leq i_k\leq N}.
\end{align*}
Here is the passage between them.

\begin{proposition}\label{prop20}
Let us fix $n\geq 1$. For all $I=\{i_1<\ldots<i_k\} \subseteq [n]$, we put:
\begin{align*}
\lambda(I)&=(i_1-1)\ldots (i_k-k),&
\mu(I)&=(-1)^k(i_1-1)i_2(i_3+1)\ldots (i_k+k-2).
\end{align*}
We use the following notation: if $[n]\setminus I=\{q_1<\ldots<q_l\}$,
$\displaystyle \prod_{q\notin I}^\blacktriangleleft \epsilon_{i_q}=\epsilon_{i_{q_1}}\blacktriangleleft\ldots \blacktriangleleft \epsilon_{i_{q_l}}$. Then:
\begin{align*}
\epsilon_{i_1}\blacktriangleleft \ldots \blacktriangleleft \epsilon_{i_n}&=\sum_{I\subseteq [n]} \lambda(I)\left(\prod_{p\in I} a_{i_p} \right)
\left(\prod_{q\notin I} \epsilon_{i_q}\right),\\
\epsilon_{i_1}\ldots \epsilon_{i_n}&=\sum_{I\subseteq [n]} \mu(I)\left(\prod_{p\in I} a_{i_p} \right)
\left(\prod_{q\notin I}^\blacktriangleleft \epsilon_{i_q}\right).
\end{align*}
\end{proposition}

\begin{proof} \textit{First step.} Let us prove the first formula by induction on $n$. It is obvious if $n=1$, as $\lambda(\emptyset)=1$ and $\lambda(\{1\})=0$. 
Let us assume the result at rank $n$.
\begin{align*}
\epsilon_{i_1}\blacktriangleleft\ldots \blacktriangleleft\epsilon_{i_{n+1}}&=\sum_{I\subseteq [n]} \lambda(I)
\left(\prod_{p\in I} a_{i_p}\right) \left(\prod_{q\notin I} \epsilon_{i_q}\right)\blacktriangleleft \epsilon_{i_{n+1}}\\
&=\sum_{I\subseteq [n]} \lambda(I)\left(\prod_{p\in I} a_{i_p}\right) \left(\prod_{q\notin I} \epsilon_{i_q} \epsilon_{i_{n+1}}
+(k-|I|) a_{i_{n+1}}\prod_{q\notin I} \epsilon_{i_q}\right)\\
&=\sum_{\substack{I\subseteq [n+1],\\n+1\notin I}} \lambda(I)\left(\prod_{p\in I} a_{i_p} \right)\left(\prod_{q\notin I} \epsilon_{i_p}\right)
+\sum_{\substack{I\subseteq [n+1],\\n+1\in I}} \lambda(I)\left(\prod_{p\in I} a_{i_p} \right)\left(\prod_{q\notin I} \epsilon_{i_p}\right)\\
&=\sum_{I\subseteq [n+1]} \lambda(I)\left(\prod_{p\in I} a_{i_p} \right)\left(\prod_{q\notin I} \epsilon_{i_p}\right).
\end{align*}

\textit{Second step.} Let us prove that for all $I\subseteq [n]$, $\displaystyle \sum_{J\subseteq I} \lambda(J)\mu(I\setminus J)=\delta_{I,\emptyset}$.

We put $I=\{i_1<\ldots<i_k\}$ and we proceed by induction on $k$. As $\lambda(\emptyset)=\mu(\emptyset)=1$, 
the result is obvious at rank $k=0$ and $k=1$. Let us assume the result at rank $k-1$, with $k\geq 2$. 
\begin{align*}
\sum_{J \subseteq I} \lambda(J)\mu(I\setminus J)&=\sum_{\substack{J \subseteq I,\\ i_k\in J}} \lambda(J)\mu(I\setminus J)
+\sum_{\substack{J \subseteq I,\\ i_k\notin J}} \lambda(J)\mu(I\setminus J)\\
&=\sum_{J\subseteq I\setminus\{i_k\}}\lambda(J\cup \{i_k\})\mu(I\setminus \{i_k\}\setminus J)
+\sum_{J\subseteq I\setminus\{i_k\}}\lambda(J)\mu(I\setminus J)\\
&=\sum_{J\subseteq I\setminus\{i_k\}}\lambda(J)(i_k-|J|)\mu(I\setminus\{i_k\}\setminus J)\\
&-\sum_{J\subseteq I\setminus\{i_k\}}\lambda(J)\mu(I\setminus \{i_k\}\setminus J)(i_k+|I\setminus \{i_k\}\setminus J|+1)\\
&=\sum_{J\subseteq I\setminus\{i_k\}}\lambda(J)\mu(I\setminus \{i_k\}\setminus J)(i_k-|J|-i_k-|I|+1+|J|-1)\\
&=-|I|\sum_{J\subseteq I\setminus\{i_k\}}\lambda(J)\mu(I\setminus \{i_k\}\setminus J)\\
&=0.
\end{align*}
Therefore:
\begin{align*}
\sum_{I\subseteq [n]} \mu(I) \left(\prod_{p\in I} a_{i_p}\right)\left(\prod_{q\notin I}^\blacktriangleleft \epsilon_{i_q}\right)
&=\sum_{I\subseteq [n]}\sum_{J\subseteq [n]\setminus I} \mu(I)\lambda(J) \left(\prod_{p\in I} a_{i_p}\right)
\left(\prod_{p\in J} a_{i_p}\right)\left(\prod_{q\in [n]\setminus I\setminus J} \epsilon_{i_q}\right)\\
&=\sum_{A\sqcup B\sqcup C=[n]} \mu(A)\lambda(B) \left(\prod_{p\in A\sqcup B} a_{i_p}\right) \left(\prod_{q\in C} \epsilon_{i_q}\right)\\
&=\sum_{I\sqcup J=[n]}\left(\underbrace{\sum_{I'\subseteq I} \lambda(I')\mu(I\setminus I')}_{=\delta_{I,\emptyset}}\right)
\left(\prod_{p\in I}a_{i_p}\right) \left(\prod_{q\in J} \epsilon_{i_q}\right)\\
%&=\prod_{q\in [n]}\epsilon_{i_q}+0\\
&=\epsilon_{i_1}\ldots \epsilon_{i_n},
\end{align*}
which ends the proof. \end{proof}

\subsection{Modules over $\g_{(1,0,\ldots,0)}$}

\begin{proposition}\label{prop21}\begin{enumerate}
\item Let $V$ be a module over the associative (non unitary) algebra $(\g_{(1,0,\ldots,0)},\triangleleft)$. Then $V=V^{(0)}\oplus V^{(1)}$, with:
\begin{itemize}
\item $\epsilon_1.v=v$ if $v\in V^{(1)}$ and $\epsilon_1.v=0$ if $v\in V^{(0)}$.
\item For all $i\geq 2$, $\epsilon_i.v\in V^{(0)}$ if $v\in V^{(1)}$ and $\epsilon_i.v=0$ if $i\in V^{(0)}$.
\end{itemize}
\item Conversely, let $V=V^{(1)}\oplus V^{(0)}$ be a vector space and let $f_i:V^{(1)}\longrightarrow V^{(0)}$ for all $2\leq i\leq N$.
One defines a structure of $(\g_{(1,0,\ldots,0)},\triangleleft)$-module over $V$:
\begin{align*}
\epsilon_1.v&=\begin{cases}
v\mbox{ if }v\in V^{(1)},\\
0\mbox{ if }v\in V^{(0)};
\end{cases}&
\mbox{ if }i\geq 2,\: \epsilon_i.v&=\begin{cases}
f_i(v)\mbox{ if }v\in V^{(1)},\\
0\mbox{ if }v\in V^{(0)}.
\end{cases}&
\end{align*}
Shortly:
\begin{align*}
\epsilon_1&: \left[\begin{array}{cc}
0&0\\
0&Id
\end{array}\right],&
\forall i\geq 2,\:\epsilon_i&: \left[\begin{array}{cc}
0&f_i\\
0&0
\end{array}\right].
\end{align*}\end{enumerate}\end{proposition}

\begin{proof} Note that in $\g_{(1,0,\ldots,0)}$, $\epsilon_i\triangleleft \epsilon_j=\delta_{1,j} \epsilon_i$. \\

1. In particular, $\epsilon_1\triangleleft \epsilon_1=\epsilon_1$. If $F_1:V\longrightarrow V$ is defined by $F_1(v)=\epsilon_1.v$,
then:
$$F_1\circ F_1(v)=\epsilon_1.(\epsilon_1.v)=(\epsilon_1\triangleleft \epsilon_1).v=\epsilon.v=F_1(v),$$
so $F_1$ is a projection, which implies the decomposition of $V$ as $V^{(0)}\oplus V^{(1)}$. Let $x \in V^{(1)}$ and $i\geq 2$. Then
$F_1(\epsilon_i.v)=\epsilon_1.(\epsilon_i.v)=(\epsilon_1\triangleleft \epsilon_i).v=0$, so $\epsilon_i.v \in V^{(0)}$. Let $x \in V^{(0)}$. Then
$\epsilon_i.v=(\epsilon_i\triangleleft \epsilon_1).v=\epsilon_i.F_1(v)=0$, so $\epsilon_i.v=0$.\\

2. Let $i \geq 2$ and $j\in [N]$.  If $v\in V^{(1)}$:
\begin{align*}
\epsilon_1.(\epsilon_1.v)&=v=\epsilon_1.v,&\epsilon_i.(\epsilon_1.v)&=f_i(v)=\epsilon_i.v,&\epsilon_j.(\epsilon_i.v)&=\epsilon_j.f_i(v)=0.v.
\end{align*}
If $v\in V^{(0)}$:
\begin{align*}
\epsilon_1.(\epsilon_1.v)&=0=\epsilon_1.v,&\epsilon_i.(\epsilon_1.v)&=0=\epsilon_i.v,&\epsilon_j.(\epsilon_i.v)&=0=0.v.
\end{align*}
So $V$ is indeed a $(\g_{(1,0,\ldots,0)},\triangleleft)$-module. \end{proof}\\

\textbf{Example.} There are, up to an isomorphism, three indecomposable $(\g_{(1,0)},\triangleleft)$-modules:

$$\begin{array}{c|c|c|c}
\epsilon_1&(0)& (1)&\left(\begin{array}{cc}0&0\\0&1\end{array}\right)\\
\hline \epsilon_2&(0)&(0)&\left(\begin{array}{cc}0&1\\0&0\end{array}\right)
\end{array}$$

\begin{proposition} \label{prop22} (We assume $\K$ algebraically closed).
Let $V$ be an indecomposable finite-dimensional module over the Lie algebra $\g_{(1,0,\ldots,0)}$. There exists a scalar $\lambda$
and a decomposition: 
$$V=V^{(0)}\oplus\ldots \oplus V^{(k)}$$
such that, for all $0\leq p\leq k$:
\begin{itemize}
\item $\epsilon_1\left(V^{(p)}\right) \subseteq V^{(p)}$ and there exists $n\geq 1$ such that $(\epsilon_1-(\lambda+p)Id)^n_{\mid V^{(p)}}=(0)$.
\item If $i\geq 2$, $\epsilon_i\left(V^{(p)}\right)\subseteq V^{(p-1)}$, with the convention $V^{(-1)}=(0)$.
\end{itemize}
\end{proposition}

\begin{proof} First, observe that in the enveloping algebra of $\g_{(1,0,\ldots,0)}$, if $i\geq 2$ and $\lambda \in \K$:
\begin{align*}
\epsilon_i\blacktriangleleft (\epsilon_1-\lambda)&=\epsilon_i\epsilon_1+\epsilon_i-\lambda \epsilon_i
=\epsilon_i\epsilon_1+(1-\lambda) \epsilon_i=(\epsilon_1-\lambda+1)\blacktriangleleft \epsilon_i.
\end{align*}
Therefore, for all $i\geq 2$, for all $n \in \N$, for all $\lambda \in \K$:
$$\epsilon_i\blacktriangleleft (\epsilon_1-\lambda)^{\blacktriangleleft n}=(\epsilon_1-\lambda+1)^{\blacktriangleleft n}\blacktriangleleft \epsilon_i.$$

Let $V$ be a finite-dimensional module over the Lie algebra $\g_{(1,0,\ldots,0)}$. We denote by $E_\lambda$ the characteristic subspace of eigenvalue
$\lambda$ for the action of $\epsilon_1$. Let us prove that for all $\lambda \in \K$, if $i\geq 2$, $\epsilon_i(E_\lambda)\subseteq E_{\lambda-1}$.
If $x\in E_\lambda$, there exists $n \geq 1$, such that $(\epsilon_1-\lambda Id)^{\blacktriangleleft n}.v=0$. Hence:
$$0=\epsilon_i.((\epsilon_1-\lambda Id)^n.v)=(\epsilon_1-(\lambda-1)Id)^n.(\epsilon_i.v),$$
so $\epsilon_i \in E_{\lambda-1}$.

Let us take now $V$ an indecomposable module, and let $\Lambda$ be the spectrum of the action of $\epsilon_1$. 
The group $\mathbb{Z}$ acts on $\K$ by translation. We consider $\Lambda'=\Lambda+\mathbb{Z}$ and let $\Lambda''$ be a system of representants
of the orbits of $\Lambda'$. Then:
$$V=\bigoplus_{\lambda \in \Lambda''} \underbrace{\left(\bigoplus_{n\in \mathbb{Z}} E_{\lambda+n}\right)}_{V_\lambda}.$$
By the preceding remarks, $V_\lambda$ is a module. As $V$ is indecomposable, $\Lambda''$ is reduced to a single element.
As the spectrum of $\epsilon_1$ is finite, it is included in a set of the form $\{\lambda,\lambda+1,\ldots,\lambda+k\}$.
We then take $V^{(p)}=E_{\lambda+p}$ for all $p$. \end{proof}\\

\textbf{Example.} Let us give the indecomposable modules of $\g_{(1,0)}$ of dimension $\leq 3$. For any $\lambda \in \K$:
\begin{align*}&\begin{array}{c|c}
\epsilon_1&\epsilon_2\\
\hline (\lambda)&(0)\\
\hline \left(\begin{array}{cc}\lambda&0\\0&\lambda+1\end{array}\right)&\left(\begin{array}{cc}0&1\\0&0\end{array}\right)\\
\hline \left(\begin{array}{cc}\lambda&1\\0&\lambda\end{array}\right)&\left(\begin{array}{cc}0&0\\0&0\end{array}\right)\\
\hline \left(\begin{array}{ccc}\lambda&0&0\\0&\lambda+1&0\\0&0&\lambda+2\end{array}\right)
&\left(\begin{array}{ccc}0&1&0\\0&0&1\\0&0&0\end{array}\right)
\end{array}&&\begin{array}{c|c}
\epsilon_1&\epsilon_2\\
\hline \left(\begin{array}{ccc}\lambda&1&0\\0&\lambda&0\\0&0&\lambda+1\end{array}\right)&
\left(\begin{array}{ccc}0&0&1\\0&0&0\\0&0&0\end{array}\right)\\
\hline \left(\begin{array}{ccc}\lambda&0&0\\0&\lambda+1&1\\0&0&\lambda+1\end{array}\right)&
\left(\begin{array}{ccc}0&1&0\\0&0&0\\0&0&0\end{array}\right)\\
\hline \left(\begin{array}{ccc}\lambda&1&0\\0&\lambda&1\\0&0&\lambda\end{array}\right)
&\left(\begin{array}{ccc}0&0&0\\0&0&0\\0&0&0\end{array}\right)
\end{array}\end{align*}

\begin{definition}
Let $V$ be a module over the Lie algebra $\g_a$. The associated algebra morphism is:
$$\phi_V:\left\{\begin{array}{rcl}
\mathcal{U}(\g_a)=(S(\g_a),\blacktriangleleft)&\longrightarrow&End(V)\\
\epsilon_i&\longrightarrow&\left\{\begin{array}{rcl}
V&\longrightarrow&V\\
v&\longrightarrow&\epsilon_i.v.
\end{array}\right.
\end{array}\right.$$
For all $i_1,\ldots,i_k \in [N]$, we put $F_{i_1,\ldots,i_k}=\phi_V(\epsilon_{i_1}\ldots \epsilon_{i_k})$; 
this does not depend on the order on the indices $i_p$.
\end{definition}

By Proposition \ref{prop20}:

\begin{proposition}
For all $i_1,\ldots,i_n \in [N]$:
\begin{align*}
F_{i_1}\circ \ldots \circ F_{i_n}&=\sum_{\substack{I\subseteq [n], \\ I\setminus J=\{j_1<\ldots<j_l\}}} \lambda(I) \left(\prod_{p\in I} a_{i_p}\right)
F_{i_{j_1},\ldots,i_{j_l}},\\
F_{i_1,\ldots,i_n}&=\sum_{\substack{I\subseteq [n], \\ I\setminus J=\{j_1<\ldots<j_l\}}} \mu(I) \left(\prod_{p\in I} a_{i_p}\right)
F_{i_{j_1}}\circ \ldots \circ F_{i_{j_l}}.
\end{align*}
\end{proposition}

When $V$ is a module over the associative algebra $(\g_A,\triangleleft)$, these morphisms are easy to describe:

\begin{proposition}\label{prop25}
Let $V$ be a module over the associative algebra $(\g_a,\triangleleft)$; it is also a module over the Lie algebra $(\g_a,[-,-]_a)$. 
For all $k\geq 2$, $i_1,\ldots,i_k \in [N]$, $F_{i_1,\ldots,i_k}=0$.
\end{proposition}

\begin{proof} As $V$ is a module over the associative algebra $(\g_a,\triangleleft)$, for any $i_1,i_2\in [N]$:
$$F_{i_1}\circ F_{i_2}=a_{i_2} F_{i_1}.$$
We proceed by induction on $k$. If $k=2$, $\epsilon_{i_1}\epsilon_{i_2}=\epsilon_{i_1} \blacktriangleleft \epsilon_{i_2}-a_{i_2} \epsilon_{i_1}$, so:
$$F_{i_1,i_2}=F_{i_1} \circ F_{i_2}-a_{i_2}F_{i_1}=a_{i_2}F_{i_1}-a_{i_2}F_{i_1}=0.$$
Let us assume the result at rank $k$. Then $\epsilon_1\ldots \epsilon_{i_{k+1}}=\epsilon_{i_1}\ldots \epsilon_{i_k} \blacktriangleleft \epsilon_{i_{k+1}}
-k a_{i_{k+1}} \epsilon_{i_1}\ldots \epsilon_{i_k}$, 
and $F_{i_1,\ldots,i_{k+1}}=F_{i_1,\ldots,i_k} \circ F_{i_{k+1}}-ka_{i_{k+1}} F_{i_1,\ldots,i_k}=0$. \end{proof}

\section{A family of post-Lie algebras}

\subsection{Reminders}

We defined in \cite{Foissy2} a family of pre-Lie algebras, associated to endomorphisms. Let us briefly recall this construction.

\begin{proposition}
Let $V$ be a vector space and $F:V\longrightarrow V$ be an endomorphism. We define a product $*$ on $T(V)$:
for all $f,g \in T(V)$, for all $x\in V$,
\begin{align*}
\emptyset*g&=0,&
xf*g&=x(f*g)+F(x)(f\shuffle g).
\end{align*}
This product is pre-Lie. The pre-Lie algebra $(T(V),*)$ is denoted by $T(V,F)$. Moreover, for all $f,g,h\in T(V,F)$:
$$(f\shuffle g)*h=(f*h)\shuffle g+f\shuffle (g*h).$$
\end{proposition}

We also proved the following result:

\begin{proposition}\label{prop27}
Let $k,l \geq 0$. 
\begin{itemize}
\item The set $Sh(k,l)$ of $(k,l)$-shuffles is the set of permutation $\sigma \in \mathfrak{S}_{k+l}$ such that
$\sigma(1)<\ldots<\sigma(k)$ and $\sigma(k+1)<\ldots <\sigma(k+l)$.
\item If $\sigma\in Sh(k,l)$, we put $m_k(\sigma)=\max\{i\in [k]\mid \sigma(1)=1,\ldots,\sigma(i)=i\}$.
In particular, if $\sigma(1)\neq 1$, $m_k(\sigma)=0$. 
\end{itemize}
For all $x_1,\ldots,x_k$, $y_1,\ldots,y_l\in V$:
\begin{align*}
x_1\ldots x_k*y_1\ldots y_l&=\sum_{\sigma \in Sh(k,l)} \sum_{p=1}^{m_k(\sigma)} \left(Id^{\otimes(p_1)}\otimes F\otimes Id^{\otimes(k+l-p)}\right)
\sigma.(x_1\ldots x_ky_1\ldots y_l).
\end{align*}\end{proposition}

\subsection{Construction}

Let us fix a vector space $V$, a family of $N$ endomorphisms $(F_1,\ldots,F_N)$ of $V$ and $a=(a_1,\ldots,a_N) \in \K^N$.
We define inductively a product $*$ on $T(V)^N$: for all $f,g \in T(V)^N$, $x\in V$, $i\in [N]$,
\begin{align*}
\emptyset \epsilon_i*g&=0,&
xf*g&=x(f*g)+F_1(x)(f \shuffle_1 g)+\ldots+F_N(x)(f\shuffle_N g).
\end{align*}
We define a second product $\bullet$ on $T(V)^N$: 
$$\forall f,g \in T(V)^N,\: f\bullet g=f*g+ f\shufg{a} g.$$

\textbf{Examples.} Let $x,y,z\in V$, $g\in T(V)$, $i,j \in [N]$. Then:
\begin{align*}
x\epsilon_i*g\epsilon_j&=F_j(x)g\epsilon_j,\\
xy\epsilon_i*g\epsilon_j&=(xF_j(y)g+F_j(x)(y\shuffle g)) \epsilon_i,\\
xyz\epsilon_i*g\epsilon_j&=(xyF_j(z)g+xF_j(y)(z\shuffle g)+F_j(x)(yz\shuffle g))\epsilon_i.
\end{align*}

\begin{proposition}\label{prop28}
Let $x_1,\ldots,x_k,y_1,\ldots,y_l \in V$, $i,j \in [N]$.
\begin{align*}
x_1\ldots x_k \epsilon_i*y_1\ldots y_l\epsilon_j&=\sum_{\sigma \in Sh(k,l)} \sum_{p=1}^{m_k(\sigma)} \left(Id^{\otimes(p_1)}\otimes F_j 
\otimes Id^{\otimes(k+l-p)}\right)\sigma.(x_1\ldots x_ky_1\ldots y_l)\epsilon_i.
\end{align*}\end{proposition}

\begin{proof} By induction on $k$. It is immediate if $k=0$, as both sides are equal to $0$. Let us assume the result at rank $k-1$.
\begin{align*}
x_1\ldots x_k \epsilon_i*y_1\ldots y_l\epsilon_j&=x_1(x_2\ldots x_k\epsilon_i*y_1\ldots y_l\epsilon_j)+F_j(x_1)
(x_2\ldots x_k\shuffle y_1\ldots y_l)\epsilon_i\\
&=\sum_{\substack{\sigma \in Sh(k,l),\\ \sigma(1)=1}}
\sum_{p=2}^{m_k(\sigma)} (Id^{\otimes(p_1)}\otimes F_j 
\otimes Id^{\otimes(k+l-p)})\sigma.(x_1\ldots x_ky_1\ldots y_l)\epsilon_i\\
&+\sum_{\substack{\sigma \in Sh(k,l),\\ \sigma(1)=1}}
(F_j \otimes Id^{\otimes(k+l-1)})\sigma.(x_1\ldots x_ky_1\ldots y_l)\epsilon_i\\
&=\sum_{\substack{\sigma \in Sh(k,l),\\ \sigma(1)=1}}\sum_{p=1}^{m_k(\sigma)} (Id^{\otimes(p_1)}\otimes F_j 
\otimes Id^{\otimes(k+l-p)})\sigma.(x_1\ldots x_ky_1\ldots y_l)\epsilon_i\\
&=\sum_{\sigma \in Sh(k,l)}\sum_{p=1}^{m_k(\sigma)} (Id^{\otimes(p_1)}\otimes F_j 
\otimes Id^{\otimes(k+l-p)})\sigma.(x_1\ldots x_ky_1\ldots y_l)\epsilon_i,
\end{align*}
so the result holds for all $k$. \end{proof}\\

\textbf{Remark.} Let $*_j$ be the pre-Lie product of $T(V,F_j)$, described in \cite{Foissy2}. For 	all $f,g \in T(V)$, for all $i,j \in [N]$:
$$f\epsilon_i*g\epsilon_j=(f*_j g)\epsilon_i.$$

\begin{corollary}\label{cor29}
For all $f,g,h \in T(V)^N$, for all $i\in [N]$:
\begin{align*}
(f \shufg{i} g)*h&=(f*h)\shufg{i} g+f \:\shufg{i} (g*h),\\
(f \shuffle_i g)*h&=(f*h)\shuffle_i g+f\shuffle_i (g*h),\\
(f\shuffle g)*h&=(f*h)\shuffle g+f\shuffle (g*h).
\end{align*}\end{corollary}

\begin{proof} It is enough to prove these assertions for $f=f'\epsilon_k$, $g=g'\epsilon_l$ and $h=h'\epsilon_m$, with $f',g',h'\in T(V)$.
For the first assertion:
\begin{align*}
(f \:_i \shuffle g)*h&=\delta_{i,k} (f'\shuffle g' \epsilon_l)*h'\epsilon_m\\
&=\delta_{i,k}(f'\shuffle g')*_m h'\epsilon_l\\
&=\delta_{i,k}((f'*_mh')\shuffle g'+f'\shuffle (g'*_m h'))\epsilon_l\\
&=(f*h)\shufg{i} g+f \shufg{i} (g*h).
\end{align*}
The second point is deduced from the first one, as $\shuffle_i=\shufg{i}^{op}$. Finally:
\begin{align*}
(f \:_i \shuffle g)*h&=\delta_{k,l} (f'\shuffle g' \epsilon_l)*h'\epsilon_m\\
&=\delta_{k,l}(f'\shuffle g')*_m h'\epsilon_l\\
&=\delta_{k,l}((f'*_mh')\shuffle g'+f'\shuffle (g'*_m h'))\epsilon_l\\
&=(f*h)\shuffle g+f\shuffle (g*h).
\end{align*}
So the last point holds. \end{proof}

\begin{theorem} \label{theo30}
The following conditions are equivalent:
\begin{enumerate}
\item $(T(V)^N,\bullet)$ is a pre-Lie algebra.
\item $\g'_a=(T(V)^N,\:_a[-,-],*)$ is a post-Lie algebra.
\item $V$ is a module over the Lie algebra $\g_a$, with the action given by $\epsilon_i.v=F_i(v)$.
\end{enumerate}\end{theorem}

\begin{proof}  By Corollary \ref{cor29}, for all $f,g,h \in \g'_a$, $\:_a[f,g]*h=\:_a[f*h,g]+\:_a[f,g*h]$.\\

$1.\Longleftrightarrow 2$. Let $f,g,h\in \g$.
\begin{align*}
&(f\bullet g)\bullet h-f\bullet (g\bullet h)-(f\bullet h)\bullet g+f\bullet (h\bullet g)\\
&=(f*g)*h-f*(g*h)-(f*h)*g-f*(h*g)\\
&+(f \shufg{a} g)*h-f \shufg{a}(g*h)-(f\shufg{a} h)*g+f\shufg{a}(h*g)\\
&+(f*g)\shufg{a} h-f*(g\shufg{a} h)-(f*h)\shufg{a} g+f*(h\shufg{a} g)\\
&+(f\shufg{a} g)\shufg{a} h-f\shufg{a}(g\shufg{a} h)-(f\shufg{a} g)\shufg{a} h+f\shufg{a}(g \shufg{a} h)\\
&=(f*g)*h-f*(g*h)-(f*h)*g-f*(h*g)\\
&+f*(g \shufg{a} h)-f*(h \shufg{a} g)\\
&+[(f\shufg{a} g)*h-f \shufg{a} (g*h)-(f*h)\shufg{a} g]\\
&+[(f\shufg{a} h)*g-f \shufg{a} (h*g)-(f*g)\shufg{a} h]\\
&-[(f*h)\shufg{a} g-f\shufg{a}(h*g)-(f*g)\shufg{a} h]\\
&+[(f\shufg{a} g)\shufg{a} h-f\shufg{a}(g\shufg{a} h)]-[(f\shufg{a} g)\shufg{a} h-f\shufg{a}(g \shufg{a} h)]\\
&=(f*g)*h-f*(g*h)-(f*h)*g+f*(h*g)-f*\:_a[g,h].
\end{align*}
So $(\g'_a,\bullet)$ is pre-Lie if, and only if, $(\g'_a,\:_a[-,-],*)$ is post-Lie.\\

$2.\Longrightarrow 3$. Let $x,y,v \in V$ and $i,j,k\in [N]$. Then:
\begin{align*}
x\epsilon_i*y\epsilon_j&=F_j(x)y\epsilon_i,&xy\epsilon_i*z\epsilon_k&=xF_k(y)z\epsilon_i+F_k(x)(y\shuffle z)\epsilon_i.\\
x\epsilon_i*yz\epsilon_k&=F_k(x)yz\epsilon_i,
\end{align*}
Hence:
\begin{align*}
(x\epsilon_i*y\epsilon_j)*z\epsilon_k&=F_j(x)y\epsilon_i*z\epsilon_k\\
&=F_j(x)F_k(y)z\epsilon_i+F_k\circ F_j(x)y\shuffle z\epsilon_i,\\
x\epsilon_i*(y\epsilon_j*z\epsilon_k)&=x\epsilon_i*F_k(y)z\epsilon_j\\
&=F_j(x)F_k(y)z\epsilon_i,\\
x\epsilon_i \:_a[y\epsilon_j,z\epsilon_k]&=a_jx\epsilon_i*(y\shuffle z)\epsilon_k-a_kx\epsilon_i*(y\shuffle z)\epsilon_j\\
&=(a_jF_k(x)(y\shuffle z)-a_kF_j(x)(y\shuffle z))\epsilon_i.
\end{align*}
The post-Lie relation (\ref{EQ2}) gives:
\begin{align*}
&(a_jF_k(x)-a_kF_j(x))(y\shuffle z)\\
&=F_j(x)F_k(y)z+F_k\circ F_j(x)(y\shuffle z)-F_j(x)F_k(y)z-F_j\circ F_k(x)(y\shuffle z)\\
&=(F_j\circ F_k-F_k\circ F_j)(x)(y\shuffle z).
\end{align*}
Let $y=z$ be a nonzero element of $V$. Then $y\shuffle z\neq 0$, and we obtain that  for all $x\in V$, $a_jF_k(x)-a_kF_j(x)=(F_j\circ F_k-F_k\circ F_j)(x)$: $V$ is a $\g_a$-module.\\

$3.\Longrightarrow 2$. Let us prove the post-Lie relation (\ref{EQ2}) for $f\epsilon_i$, $g$ and $h$, with $f\in T(V)$, $i\in [N]$, $g,h\in \g'_a$.
We assume that $f$ is a word and we proceed by induction on the length $n$ of $f$. If $n=0$, then $f=\emptyset$ and every term is $0$
in the relation. Let us assume the result at rank $n-1$. We put $f=xf'$, with $x\in V$, and $f'$ a word of length $n-1$.
\begin{align*}
(f*g)*h&=x((f'\epsilon_i*g)*h)+\sum_{p=1}^N F_p(x)((f'\epsilon_i*g)\shuffle_p h)\\
&+\sum_{p=1}^N F_p(x)((f'\epsilon_i \shuffle_p g)*h+\sum_{p,q=1}^N F_q\circ F_p(x)(f'\epsilon_i \shuffle_p g\shuffle_q h),\\
f*(g*h)&=x(f'\epsilon_i*(g*h))+\sum_{p=1}^N F_p(x)(f'\epsilon_i\shuffle_p (g*h)),\\
\sum_{p=1}^N a_p f*(g \:_p\shuffle h)&=\sum_{p=1}^N a_p x(f'\epsilon_i*(g\:_p\shuffle h))+\sum_{p,q=1}^N
a_pF_q(x)(f'\epsilon_i \shuffle_q (g\:_p\shuffle h)).
\end{align*} 
We put:
$$P(f,g,h)=f*(g*h)-(f*g)*h+\sum_{p=1}^N a_p f*(g \:_p\shuffle h).$$
In order to prove the post-Lie relation (\ref{EQ2}), we have to prove that $P(f,g,h)=P(f,h,g)$. First:
\begin{align*}
(f*g)*h&=x((f'\epsilon_i*g)*h)+\sum_{p=1}^N F_p(x)((f'\epsilon_i*g)\shuffle_p h)\\
&+\sum_{p=1}^N F_p(x)((f'\epsilon_i \shuffle_p g)*h+\sum_{p,q=1}^N F_q\circ F_p(x)(f'\epsilon_i\shuffle_p g\shuffle_q h),\\
f*(g*h)&=x(f'\epsilon_i*(g*h))+\sum_{p=1}^N F_p(x)(f'\epsilon_i\shuffle_p (g*h)),\\
\sum_{p=1}^N a_p f*(g \:_p\shuffle h)&=\sum_{p=1}^N a_p x(f'\epsilon_i*(g\:_p\shuffle h))+\sum_{p,q=1}^N
a_pF_q(x)(f'\epsilon_i \shuffle_q (g\:_p\shuffle h)).
\end{align*} 
Consequently:
\begin{align*}
P(f,g,h)&=xP(f',g,h)\\
&+\sum_{p=1}^N F_p(x)(-(f'\epsilon_i*g)\shuffle_p h-((f'\epsilon_i \shuffle_p g)*h+f'\epsilon_i \shuffle_p(g*h))\\
&+\sum_{p,q=1}^N a_pF_q(x)(f'\epsilon_i \shuffle_q g \:_p\shuffle h)-F_q\circ F_p(x)(f'\epsilon_i \shuffle_p g\shuffle_q h)\\
&=xP(f',g,h)-\sum_{p=1}^N F_p(x)((f'\epsilon_i*g)\shuffle_p h+(f'\epsilon_i*h)\shuffle_p g)\\
&+\sum_{p,q=1}^N a_pF_q(x)(f'\epsilon_i \shuffle_q h \shuffle_p g)-F_q\circ F_p(x)(f'\epsilon_i \shuffle_p g\shuffle_q h)\\
&=xP(f',g,h)-\sum_{p=1}^N F_p(x)(f'\epsilon_i*g)\shuffle_p h+(f'\epsilon_i*h)\shuffle_p g)\\
&+\sum_{p,q=1}^N (a_pF_q(x)-F_q\circ F_p(x))(f'\epsilon_i \shuffle_p g \shuffle_q h).
\end{align*}
By the induction hypothesis, $P(f',g,h)=P(f',h,g)$, so the first row is symmetric in $g,h$. 
As $V$ is a $\g_a$-module, $a_pF_q-F_q\circ F_p=a_qF_p-F_p\circ F_q$, so the second row is symmetric in $g,h$, and finally
$P(f,g,h)=P(f,h,g)$: $\g'_a$ is a post-Lie algebra. \end{proof}\\

\textbf{Example.} The post-Lie algebra $\siso$ is associated to $a=(1,0)$, $V=Vect(x_1,x_2)$ and:
\begin{align*}
F_1=&\left(\begin{array}{cc}
0&0\\0&1
\end{array}\right),&
F_2&=\left(\begin{array}{cc}
0&1\\0&0
\end{array}\right).
\end{align*}
As $F_1$ and $F_2$ define a module over the Lie algebra $\g_{(1,0)}$, even in fact over the associative algebra $(\g_{(1,0)},\triangleleft)$,
we obtain indeed a post-Lie algebra. For all $f,g\in T(V)$, for all $i,j \in  \{1,2\}$:
\begin{align*}
\emptyset\epsilon_i*g\epsilon_j&=0,&x_2f\epsilon_i*g\epsilon_1&=x_2(f\epsilon_i*g\epsilon_1)+x_2(f\shuffle g)\epsilon_i,\\
x_1f\epsilon_i*g\epsilon_j&=x_1(f\epsilon_i*g\epsilon_j),&x_2f\epsilon_i*g\epsilon_2&=x_2(f\epsilon_i*g\epsilon_2)+x_1(f\shuffle g)\epsilon_i.
\end{align*}

\subsection{Extension of the post-Lie product}

We now extend the post-Lie product of $\g'_a$ to the enveloping algebra $\U(\g'_a)$.
As this Lie bracket is obtained from an associative product $\triangleleft=\shufg{a}$, we can see $\U(\g'_a)$ as $(S(\g'_a),\blacktriangleleft,\Delta)$.
The post-Lie product $*$ is extended to $\U(\g'_a)$, and we obtain a Hopf algebra $(\U(\g),\circledast,\Delta)$,
isomorphic to $\U(\g'_a,\:_a[-,-]_*)$, with:
$$\forall f,g \in \g,\: \:_a[f,g]_*=\:_a[f,g]+f*g-g*f=f\shufg{a} g+f*g-g\shufg{a} f-g*f.$$
As $\bullet$ is a pre-Lie product, it can also be extended to $S(\g)$ and gives a product $\odot$, making $S(\g'_a)$ a Hopf algebra
isomorphic to $\U(\g'_a,[-,-]_\bullet)$. \\

\textbf{Remark.} Let $f,g \in \g'_a$.
\begin{align*}
[f,g]_\bullet&=f\bullet g-g\bullet f\\
&=f \shufg{a} g+f*g-g \shufg{a} f-g*f\\
&=\:_a[f,g]+f*g-g*f\\
&=\:_a[f,g]_*.
\end{align*}
So $[-,-]_\bullet=\:_a[-,-]_*$.

\begin{lemma} 
Let $f_1,\ldots,f_k,g\in \g'_a$, $k\geq 1$. 
\begin{align*}
(f_1\blacktriangleleft\ldots\blacktriangleleft f_k)*g&=\sum_{p=1}^k f_1\blacktriangleleft\ldots \blacktriangleleft (f_p*g)\blacktriangleleft\ldots\blacktriangleleft f_k,\\
(f_1\ldots f_k)*g&=\sum_{p=1}^k f_1\ldots (f_p*g)\ldots f_k.
\end{align*}\end{lemma}

\begin{proof} The first point comes by the very definition of $*$. For the second point, we proceed by induction on $k$. 
This is obvious if $k=1$. Let us assume the result at rank $k$, $k\geq 1$. Observe that:
$$f_1\ldots f_{k+1}=f_1\ldots f_k\blacktriangleleft f_{k+1}-\sum_{p=1}^k f_1\ldots (f_p \shufg{a} f_{k+1})\ldots f_k,$$
so:
\begin{align*}
f_1\ldots f_{k+1}*g&=(f_1\ldots f_k*g)\blacktriangleleft f_{k+1}
+f_1\ldots f_k\blacktriangleleft (f_{k+1}*g)-\sum_{p=1}^k f_1\ldots (f_p \shufg{a} f_{k+1})\ldots f_k*g\\
&=\sum_{p=1}^k f_1\ldots (f_p*g)\ldots f_k \blacktriangleleft f_{k+1}+f_1\ldots f_k \blacktriangleleft (f_{k+1}*g)\\
&-\sum_{p\neq q} f_1\ldots (f_p \shufg{a} f_{k+1})\ldots (f_q*g)\ldots f_k-\sum_{p=1}^k f_1\ldots ((f_p \shufg{a} f_{k+1})*g)\ldots f_k\\
&=\sum_{p=1}^k f_1\ldots (f_p*g)\ldots f_k f_{k+1}+\sum_{p\neq q} f_1\ldots (f_p \shufg{a} f_{k+1})\ldots (f_q*g)\ldots f_k\\
&+\sum_{p=1}^kf_1\ldots ((f_p*g) \shufg{a} f_{k+1})\ldots f_k-\sum_{p\neq q} f_1\ldots (f_p \shufg{a} f_{k+1})\ldots (f_q*g)\ldots f_k\\
&-\sum_{p=1}^kf_1\ldots ((f_p*g) \shufg{a} f_{k+1})\ldots f_k-\sum_{p=1}^k f_1\ldots (f_p\shufg{a}(f_{k+1}*g))\ldots f_k\\
&+f_1\ldots f_k(f_{k+1}*g)+\sum_{p=1}^k f_1\ldots (f_p \shufg{a}(f_{k+1}*g))\ldots f_k\\
&=\sum_{p=1}^{k+1} f_1\ldots (f_p*g)\ldots f_{k+1}.
\end{align*}
Finally, the result holds for all $k\geq 1$. \end{proof}\\

The following result allows to compute $f*g_1\ldots g_k$ by induction on the length of $f$:

\begin{proposition} \label{prop32}
Let $x\in V$, $k \geq 1$, $f,g_1,\ldots,g_k \in T(V)^N$, $i\in [N]$.
\begin{align*}
\emptyset\epsilon_i*(g_1\blacktriangleleft\ldots \blacktriangleleft g_k)&=0,\\
xf*(g_1\blacktriangleleft\ldots \blacktriangleleft g_k)&=\sum_{\substack{I=\{i_1<\ldots<i_l\}\subseteq [k],\\ j_1,\ldots,j_l\in [N]}}
F_{j_l}\circ \ldots \circ F_{j_1}(x)\left(\left(f*\prod_{i\notin I}^\blacktriangleleft g_i\right)\shufd{j_1} g_{i_1}\ldots \shufd{j_l} g_{i_l}\right);\\
\emptyset\epsilon_i*(g_1\ldots g_k)&=0,\\
xf*(g_1\ldots g_k)&=\sum_{\substack{I=\{i_1<\ldots<i_l\}\subseteq [k],\\ j_1,\ldots,j_l\in [N]}}
F_{j_1,\ldots,j_l}(x)\left(\left(f*\prod_{i\notin I} g_i\right)\shufd{j_1} g_{i_1}\ldots \shufd{j_l} g_{i_l}\right).
\end{align*}\end{proposition}

\begin{proof} In order to ease the redaction, we put:
$$I_k=\{(I,j_1,\ldots,j_l)\mid I=\{i_1<\ldots<i_l\}\subseteq [k], \: j_1,\ldots,j_l \in [N]\}.$$
We proceed by induction on $k$. It is immediate if $k=1$. Let us assume the result at rank $k$, $k\geq 1$. Then:
\begin{align*}
\emptyset\epsilon_i*(g_1\blacktriangleleft\ldots \blacktriangleleft g_{k+1})
&=(\emptyset\epsilon_i*(g_1\blacktriangleleft\ldots \blacktriangleleft g_k))*g_{k+1}\\
&-\sum_{p=1}^k\emptyset\epsilon_i*(g_1\blacktriangleleft\ldots \blacktriangleleft (g_p*g_{k+1})\blacktriangleleft\ldots\blacktriangleleft g_k)\\
&=0.
\end{align*}
Moreover:
\begin{align*}
&xf*(g_1\blacktriangleleft\ldots \blacktriangleleft g_{k+1})\\
&=( xf*(g_1\blacktriangleleft\ldots \blacktriangleleft g_k))*g_{k+1}
-\sum_{p=1}^k xf*(g_1\blacktriangleleft\ldots \blacktriangleleft (g_p*g_{k+1})\blacktriangleleft\ldots\blacktriangleleft g_k\\
&=\sum_{I_k} F_{j_l}\circ \ldots \circ F_{j_1}(x)\left(\left( f*\prod_{i\notin J\sqcup\{k+1\}}^\blacktriangleleft g_i\right)
\shufd{j_1} g_{i_1}\ldots \shufd{j_l} g_{i_l}\right)*g_{k+1}\\
&-\sum_{I_k} F_{j_l}\circ \ldots \circ F_{j_1}(x)\left( f*\left(\left(\prod_{i\notin J\sqcup\{k+1\}}^\blacktriangleleft g_i\right)*g_{k+1}\right)
\shufd{j_1} g_{i_1}\ldots \shufd{j_l} g_{i_l}\right)\\
&-\sum_{p=1}^k\sum_{I_k} F_{j_l}\circ \ldots \circ F_{j_1}(x)\left(\left( f*\prod_{i\notin J\sqcup\{k+1\}}^\blacktriangleleft g_i\right)
\shufd{j_1} g_{i_1}\ldots \shufd{j_p} (g_{i_p}*g_{k+1})\ldots \shufd{j_l} g_{i_l}\right)\\
&=\sum_{I_k}\sum_{j_{l+1}\in [N]}
F_{j_l}\circ \ldots \circ F_{j_1}(x)\left(\left( f*\prod_{i\notin J\sqcup\{k+1\}}^\blacktriangleleft g_i\right)
\shufd{j_1} g_{i_1}\ldots \shufd{j_l} g_{i_l}\shufd{j_{l+1}} g_{k+1}\right)\\
&+\sum_{I_k} F_{j_l}\circ \ldots \circ F_{j_1}(x)\\
&\hspace{1cm}\left(\left(\left(f*\prod_{i\notin J\sqcup\{k+1\}}^\blacktriangleleft g_i\right)*g_{k+1}-
f*\left(\left(\prod_{i\notin J\sqcup\{k+1\}}^\blacktriangleleft g_i\right)*g_{k+1}\right)\right)\shufd{j_1} g_{i_1}\ldots \shufd{j_l} g_{i_l}\right)\\
&=\sum_{I_k}\sum_{j_{l+1}\in [N]}F_{j_l}\circ \ldots \circ F_{j_1}(x)\left(\left( f*\prod_{i\notin J\sqcup\{k+1\}}^\blacktriangleleft g_i\right)
\shufd{j_1} g_{i_1}\ldots \shufd{j_l} g_{i_l}\shufd{j_{l+1}} g_{k+1}\right)\\
&+\sum_{I_k}F_{j_l}\circ \ldots \circ F_{j_1}(x)\left(\left( f*\prod_{i\notin J\sqcup\{k+1\}}^\blacktriangleleft g_i\blacktriangleleft g_{k+1}\right)
\shufd{j_1} g_{i_1}\ldots \shufd{j_l} g_{i_l}\shufd{j_l} g_{k+1}\right)\\
&=\sum_{I_{k+1}, \: k+1\in I} F_{j_l}\circ \ldots \circ F_{j_1}(x)\left(\left(f*\prod_{i\notin I}^\blacktriangleleft g_i\right)\shufd{j_1} 
g_{i_1}\ldots \shufd{j_l} g_{i_l}\right)\\
&+\sum_{I_{k+1}, \: k+1\notin I} F_{j_l}\circ \ldots \circ F_{j_1}(x)\left(\left(f*\prod_{i\notin I}^\blacktriangleleft g_i\right)\shufd{j_1} 
g_{i_1}\ldots \shufd{j_l} g_{i_l}\right)\\
&=\sum_{I_{k+1}} F_{j_l}\circ \ldots \circ F_{j_1}(x)\left(\left(f*\prod_{i\notin I}^\blacktriangleleft g_i\right)\shufd{j_1} 
g_{i_1}\ldots \shufd{j_l} g_{i_l}\right).
\end{align*}

So, for all $F\in \U(\g)_+$, $\emptyset\epsilon_i*F=0$.
As $g_1\ldots g_k \in \U(\g)_+$, the first point holds. Let us prove the second point by induction on $k$. The result is immediate if $k=1$.
Let us assume the result at rank $k\geq 1$.
\begin{align*}
&xf*g_1\ldots g_{k+1}\\
&=xf*(g_1\ldots g_k*g_{k+1})-\sum_{p=1}^k xf*(g_1\ldots (g_p \shufg{a} g_{k+1})\ldots g_k)\\
&=(xf*g_1\ldots g_k)*g_{k+1}\sum_{p=1}^k xf*(g_1\ldots (g_p \shufg{a} g_{k+1})\ldots g_k)-\sum_{p=1}^k xf*(g_1\ldots (g_p*g_{k+1})\ldots g_k)\\
&=\sum_{I_k} F_{j_1,\ldots,j_l}(x)\left(\left(\left(f*\prod_{i\notin J\sqcup\{k+1\}} g_i\right)*g_{k+1}\right)\shufd{j_1}g_{i_1}\ldots \shufd{j_l}g_{i_l}\right)\\
&+\sum_{I_k}\sum_{p=1}^k F_{j_1,\ldots,j_l}(x)\left(\left(f*\prod_{i\notin J\sqcup\{k+1\}} g_i\right)\shufd{j_1}g_{i_1}\ldots 
\shufd{j_p}(g_{i_p}*g_{k+1})\ldots \shufd{j_l}g_{i_l}\right)\\
&+\sum_{I_k}\sum_{j_{l+1} \in [N]} F_{j_{l+1}}\circ  F_{j_1,\ldots,j_l}(x)\left(\left(f*\prod_{i\notin J\sqcup\{k+1\}} g_i\right)\shufd{j_1}g_{i_1}\ldots 
\shufd{j_l}g_{i_l}\right)\\
&-\sum_{I_k}F_{j_1,\ldots,j_l}(x)\left(f*\left(\left(\prod_{i\notin J\sqcup\{k+1\}} g_i\right)*g_{k+1}\right)\shufd{j_1}g_{i_1}\ldots 
\shufd{j_l}g_{i_l}\right)\\
&-\sum_{I_k}\sum_{p=1}^k F_{j_1,\ldots,j_l}(x)\left(\left(f*\prod_{i\notin J\sqcup\{k+1\}} g_i\right)\shufd{j_1}g_{i_1}\ldots 
\shufd{j_p}(g_{i_p}*g_{k+1})\ldots \shufd{j_l}g_{i_l}\right)\\
&-\sum_{I_k}\sum_{p=1}^k F_{j_1,\ldots,j_l}(x)\left(\left(f*\prod_{i\notin J\sqcup\{k+1\}} g_i\right)\shufd{j_1}g_{i_1}\ldots 
\shufd{j_p}(g_{i_p}\shufg{a} g_{k+1})\ldots \shufd{j_l}g_{i_l}\right)\\
&-\sum_{I_k} F_{j_1,\ldots,j_l}(x)\left(f*\left(\left(\prod_{i\notin J\sqcup\{k+1\}} g_i\right)*g_{k+1}-\left(\prod_{i\notin J\sqcup\{k+1\}} g_i\right)g_{k+1}
\right)\shufd{j_1}g_{i_1}\ldots\shufd{j_l}g_{i_l}\right)\\
&=\sum_{I_k} F_{j_1,\ldots,j_l}(x)\left(f*\left(\left(\prod_{i\notin J\sqcup\{k+1\}} g_i\right)g_{k+1}\right)\shufd{j_1}g_{i_1}\ldots \shufd{j_l}g_{i_l}\right)\\
&+\sum_{I_k}\sum_{j_{l+1}\in [N]}F_{j_{l+1}}\circ F_{j_1,\ldots,j_l}(x)
\left(\left(f*\left(\prod_{i\notin J\sqcup\{k+1\}} g_i\right)\right)\shufd{j_1}g_{i_1}\ldots \shufd{j_l} g_{i_l} \shufd{j_{l+1}} g_{i_{l+1}}\right)\\
&-\sum_{I_k}\sum_{p=1}^k \sum_{j\in [N]}a_j F_{j_1,\ldots,j_l}(x)
\left(\left(f*\left(\prod_{i\notin J\sqcup\{k+1\}}\right)\right))\shufd{j_1}g_{i_1}\ldots \shufd{j_p} (g_{k+1} \shufd{j} g_p)\ldots \shufd{j_l}g_{i_l}\right)\\
&=\sum_{I_k} F_{j_1,\ldots,j_l}(x)\left(f*\left(\left(\prod_{i\notin J\sqcup\{k+1\}} g_i\right)g_{k+1}\right)\shufd{j_1}g_{i_1}\ldots \shufd{j_l}g_{i_l}\right)\\
&+\sum_{I_k}\sum_{j_{l+1}\in [N]}\left(F_{j_{l+1}}\circ F_{j_1,\ldots,j_l}-\sum_{p=1}^l a_{j_p} F_{j_1,\ldots,\widehat{j_p},\ldots,j_{k+1}}\right)(x)\\
&\hspace{2cm}\left(\left(f*\left(\prod_{i\notin J\sqcup\{k+1\}} g_i\right)\right)\shufd{j_1}g_{i_1}\ldots \shufd{j_l} g_{i_l} \shufd{j_{l+1}} g_{i_{l+1}}\right)
\end{align*}\begin{align*}
&=\sum_{I_{k+1},\: k+1\notin J}F_{j_1,\ldots,j_l}(x)\left(\left(f*\prod_{i\notin I} g_i\right)\shufd{j_1} g_{i_1}\ldots \shufd{j_l} g_{i_l}\right)\\
&+\sum_{I_{k+1},\: k+1\in J}F_{j_1,\ldots,j_l}(x)\left(\left(f*\prod_{i\notin I} g_i\right)\shufd{j_1} g_{i_1}\ldots \shufd{j_l} g_{i_l}\right)\\
&=\sum_{I_{k+1}}F_{j_1,\ldots,j_l}(x)\left(\left(f*\prod_{i\notin I} g_i\right)\shufd{j_1} g_{i_1}\ldots \shufd{j_l} g_{i_l}\right).
\end{align*}
Note that we used $\displaystyle F_{j_{l+1}}\circ F_{j_1,\ldots,j_l}=F_{j_1,\ldots,j_{l+1}}+\sum_{p=1}^la_{j_p} F_{j_1,\ldots,\widehat{j_p},\ldots,j_{k+1}}$.
\end{proof}

\begin{proposition}\label{prop33}
Let $k \geq 1$, $f,g_1,\ldots,g_k \in T(V)^N$. Then:
\begin{align*}
f\bullet g_1\ldots g_k&=f*g_1\ldots g_k+\sum_{p=1}^k (f*g_1\ldots g_{p-1}g_{p+1}\ldots g_k)\shufg{a} g_p.
\end{align*}\end{proposition}

\begin{proof} We proceed by induction on $k$. This is obvious if $k=1$. Let us assume the result at rank $k$, $k\geq 1$.
\begin{align*}
&f\bullet g_1\ldots g_{k+1}\\
&=(f\bullet g_1\ldots g_k)\bullet g_{k+1}-\sum_{p=1}^k f\bullet(g_1\ldots (g_p\bullet g_{k+1})\ldots g_k)\\
&=(f*g_1\ldots g_k)*g_{k+1}+(f*g_1\ldots g_k)\shufg{a} g_{k+1}\\
&+\sum_{p=1}^k ((f*g_1\ldots g_{p-1}g_{p+1}\ldots g_k)\shufg{a} g_p)*g_{k+1}
+\sum_{p=1}^k (f*g_1\ldots g_{p-1}g_{p+1}\ldots g_k)\shufg{a} g_p \shufg{a} g_{k+1}\\
&-\sum_{p=1}^k f*(g_1\ldots (g_p\bullet g_{k+1}\ldots g_k)-\sum_{p=1}^k (f*g_1\ldots g_{p-1}g_{p+1}\ldots g_k)\shuffle (g_p\bullet g_{k+1})\\
&-\sum_{p\neq q} f*(g_1\ldots (g_p\bullet g_{k+1})\ldots \widehat{g_q}\ldots g_k)\shufg{a} g_q\\
&=(f*g_1\ldots g_k)*g_{k+1}+(f*g_1\ldots g_k)\shufg{a} g_{k+1}+\sum_{p=1}^k((f*g_1\ldots g_{p-1}g_{p+1}\ldots g_k)*g_{k+1})\shufg{a} g_p\\
&+\sum_{p=1}^k(f*g_1\ldots g_{p-1}g_{p+1}\ldots g_k)\shufg{a}(g_p*g_{k+1}-g_p\bullet g_{k+1}+g_p\shufg{a} g_{k+1})\\
&-\sum_{p\neq q} f*(g_1\ldots (g_p\bullet g_{k+1})\ldots \widehat{g_q}\ldots g_k)\shufg{a} g_q-\sum_{p=1}^k f*(g_1\ldots (g_p\bullet g_{k+1})\ldots g_k)\\
&=(f*g_1\ldots g_k)*g_{k+1}+(f*g_1\ldots g_k)\shufg{a} g_{k+1}\\
&+\sum_{p=1}^k\left((f*g_1\ldots g_{p-1}g_{p+1}\ldots g_k)*g_{k+1}
-\sum_{q\neq p} f*g_1\ldots g_{p-1}g_{p+1}\ldots (g_k\bullet g_{k+1})\ldots g_k\right)\shufg{a} g_p\\
&-\sum_{p=1}^kf*g_1\ldots (g_p*g_{k+1})\ldots g_k-\sum_{p=1}^k f*g_1\ldots (g_p\shufg{a} g_{k+1})\ldots g_k\\
&=f*\left(g_1\ldots g_k \blacktriangleleft g_{k+1}-\sum_{p=1}^k g_1\ldots (g_p \shufg{a} g_{k+1})\ldots g_k\right)
+\sum_{p=1}^{k+1}(f*g_1\ldots g_{p-1}g_{p+1}\ldots g_{k+1})\shufg{a} g_p
\end{align*}\begin{align*}
&=f*g_1\ldots g_{k+1}+\sum_{p=1}^{k+1}(f*g_1\ldots g_{p-1}g_{p+1}\ldots g_{k+1})\shufg{a} g_p.
\end{align*}
So the result holds for all $k\geq 1$. \end{proof}

\begin{proposition}\label{prop34}
On $S(\g'_a)$, $\circledast=\odot$.
\end{proposition}

\begin{proof} Let $f,g\in S(\g'_a)$; let us prove that $f\circledast g=f\odot g$. We assume that $f=f_1\ldots f_k$, $g=g_1,\ldots,g_l$, 
with $f_1,\ldots, f_k,g_1,\ldots,g_l \in \g'_a$, and we proceed by induction on $k$. If $k=0$, then $f=1$ and 
$f\circledast g=f\bigcirc \hspace{-3.2mm} \bullet\hspace{1.2mm} g=g$.
Let us assume the result at all ranks $<k$. We proceed by induction on $l$. If $l=0$, then $g=1$ and $f\circledast g=f\odot g=f$.
Let us assume the result at all ranks $<l$. We put:
\begin{align*}
\Delta(f)&=f\otimes 1+1\otimes f+f'\otimes f'',&\Delta(g)&=g\otimes 1+1\otimes g+g'\otimes g''.
\end{align*}
The induction hypothesis on $k$ holds for $f'$ and $f''$ and the induction hypothesis on $l$ holds for $g'$ and $g''$.
From:
$$\Delta(f\circledast g-f\odot g)=f^{(1)}\circledast g^{(1)}\otimes f^{(2)}\circledast g^{(2)}
-f^{(1)}\odot g^{(1)}\otimes f^{(2)}\odot g^{(2)},$$
these two induction hypotheses give:
\begin{align*}
\Delta(f\circledast g-f\odot g)&=(f\circledast g-f\odot g)\otimes 1+1\otimes (f\circledast g-f\odot g).
\end{align*}
So $f\circledast g-f\odot g \in Prim(S(\g'_a))=\g'_a$. Let $\pi$ be the canonical projection on $\g'_a$ in $S(\g'_a)$. We obtain:
\begin{align*}
\pi(f\circledast g)&=\pi\left(\sum_{I\subseteq [l]} \left(f*\prod_{i\in I} g_i\right)\blacktriangleleft \prod_{j\notin I} g_j\right)\\
&=\pi\left(\sum_{[l]=I_0\sqcup \ldots \sqcup I_k}\left(f_1*\prod_{i\in I_1}g_i\right)\ldots
\left(f_k*\prod_{i\in I_k}g_i\right) \blacktriangleleft \prod_{i\in I_0} g_i\right)\\
&=\pi\left(\sum_{[l]=J_1\sqcup \ldots\sqcup J_k}\prod_{p=1}^k
\left(f_p*\prod_{i\in J_k} g_i+\sum_{j_p\in J_p}\left(f_p*\prod_{i\in J_p\setminus\{j_p\}} g_i\right)\shufg{a} g_{j_p}\right)\right)\\
&=\pi\left(\sum_{[l]=J_1\sqcup \ldots\sqcup J_k}\left(\prod_{p=1}^k f_p\bullet\prod_{i\in J_p} g_i\right)\right)\\
&=\pi\left(\left(f_1\bullet g^{(1)}\right)\ldots \left(f_k\bullet g^{(k)}\right)\right)\\
&=\pi(f\bullet g)\\
&=\pi(f\odot g).
\end{align*}
As $f\circledast g-f\odot g \in \g'_a$, $f\circledast g=f\odot g$. \end{proof}

\subsection{Graduation}

We assume in this whole paragraph that $a=(1,0,\ldots,0)$ and $V$ is finite-dimensional. 
We decompose the $\g_a$-module $V$ as a direct sum of indecomposables.
By Proposition \ref{prop22}, decomposing each indecomposables, we obtain a decomposition of $V$ of the form:
$$V=V^{(0)}\oplus \ldots \oplus V^{(k)},$$
with $F_1\left(V^{(p)}\right)\subseteq V^{(p)}$ and $F_i\left(V^{(p)}\right)\subseteq V^{(p-1)}$ for all $i\geq 2$, for all $p\in [k]$. 
We put $V_p=V^{(k+1-p)}$ for all $p\in [k+1]$. This defines a graduation of $V$, which induces a connected graduation of $T(V)$.
For this graduation of $V$, $F_1$ is homogeneous of degree $0$ and $F_i$ is homogeneous of degree $1$ for all $i\geq 2$.
We define a graduation of $\g'_a=T(V)^N$:
$$\forall n\geq 0,\: (\g'_a)_n=T(V)_n\epsilon_1\oplus \bigoplus_{i=2}^N T(V)_{n-1}\epsilon_i.$$
Let $v,w \in T(V)$, homogeneous of respective degree $k$ and $l$. Let $i,j \geq 2$.
Then:
\begin{itemize}
\item $v\epsilon_1$ is homogeneous of degree $k$.
\item $v\epsilon_i$ is homogeneous of degree $k+1$.
\item $w\epsilon_1$ is homogeneous of degree $l$.
\item $w\epsilon_j$ is homogeneous of degree $l+1$.
\end{itemize}
As $v\shuffle w$ is homogeneous of degree $k+l$:
\begin{itemize}
\item $v\epsilon_1 \:_{(1,0,\ldots,0)}\shuffle w\epsilon_1=v\shuffle w \epsilon_1$ is homogeneous of degree $k+l$.
\item $v\epsilon_1 \:_{(1,0,\ldots,0)}\shuffle w\epsilon_j=v\shuffle w \epsilon_j$ is homogeneous of degree $k+l+1$.
\item $v\epsilon_i \:_{(1,0,\ldots,0)}\shuffle w\epsilon_1=0$ is homogeneous of degree $k+l+1$.
\item $v\epsilon_i \:_{(1,0,\ldots,0)}\shuffle w\epsilon_j=0$ is homogeneous of degree $k+l+2$.
\end{itemize}
Consequently, the product $\:_{(1,0,\ldots,0)}\shuffle$ is homogeneous of degree $0$. 
Proposition \ref{prop28} implies that $*$ is homogeneous of degree $0$; summing, $\bullet$ is also homogeneous of degree $0$. Hence:

\begin{proposition}\label{prop35}
The decomposition of $V$ in indecomposable $\g_{(1,0,\ldots,0)}$-modules induces a graduation of the post-Lie algebra $\g'_{(1,0,\ldots,0)}$.
\end{proposition}

We put:
$$P(X)=\sum_{i=1}^{k+1} dim(V_p)X^p \in \K[X].$$
 the formal series of $\g'_{(1,0,\ldots,0)}$ is:
\begin{align*}
R(X)&=\sum_{p=1}^\infty dim((\g'_{(1,0,\ldots,0)})_p)X^p\\
&=\frac{1}{1-P(X)}+(N-1)\frac{X}{1-P(X)}=\frac{1+(N-1)X}{1-P(X)}.
\end{align*}
Note that $R(0)=1$: indeed, $(\g'_{(1,0,\ldots,0)})_0=Vect(\emptyset\epsilon_1)$. 
The augmentation ideal of $\g'_{(1,0,\ldots,0)}$ is:
$$(\g'_{(1,0,\ldots,0)})_+=T(V)_+\times T(V)^{N-1}.$$
This is a graded, connected post-Lie algebra.\\

\textbf{Example.} For the SISO case, $V_1=Vect(x_2)$ and $V_2=Vect(x_1)$. The formal series of $\siso$ is:
$$R_{SISO}(X)=\frac{1+X}{1-X-X^2}=1+2X+3X^2+5X^3+8X^4+13X^5+\ldots$$
Hence, $(dim(\siso)_n)_{n\geq 0}$ is the Fibonacci sequence A000045 \cite{Sloane}. For example:
\begin{align*}
(\siso)_0&=Vect(\emptyset\epsilon_1),\\
(\siso)_1&=Vect(x_2\epsilon_1,\emptyset\epsilon_2),\\
(\siso)_2&=Vect(x_1\epsilon_1,x_2x_2\epsilon_1,x_2\epsilon_2),\\
(\siso)_3&=Vect(x_1x_2\epsilon_1,x_2x_1\epsilon_1,x_2x_2x_2\epsilon_1,x_1\epsilon_2,x_2x_2\epsilon_2).
\end{align*}

\section{Graded dual}

We assume in this section that $a=(1,0,\ldots,0)$. 
The augmentation ideal of $\g'_a$ is denoted by $(\g'_a)_+$; recall that $(\g'_a)_0=Vect(\emptyset \epsilon_1)$.
\begin{itemize}
\item As $(\g'_a)_+$ is a graded, connected Lie algebra, its enveloping algebra $\U((\g'_a)_+)$ 
is a graded, connected Hopf algebra, and its graded dual also is. We denote it by $\h_V$.
\item As an algebra, $\h_V$ is identified with $S((\g'_a)^*)/\langle \emptyset\epsilon_1\rangle$.  We identify $(\g'_a)^*$ with $T(V^*)^N$ via the pairing:
$$\langle f_1\ldots f_k \epsilon_i,x_1\ldots x_l\epsilon_j\rangle=\delta_{i,j} \delta_{k,l} f_1(x_1)\ldots f_k(x_k).$$
\item The coproduct dual of $\odot=\circledast$ is denoted by $\Delta_\bullet$.
\item The dual of the product $\shufd{j}$ defined on $\g'_a$ is denoted by $\Delta_{\shufd{j}}$, defined on $(\g'_a)^*=T(V^*)^N$.
\item We define a coproduct $\Delta_*$ on $S((\g'_a)_+^*)$, dual of the right action $*$.
Therefore, this is  right coaction of $(\h_V,\Delta_\bullet)$ on itself:
$$(\Delta_*\otimes Id)\circ \Delta_*=(Id\otimes \Delta_\bullet)\circ \Delta_*.$$
\end{itemize}

\textbf{Notations}. \begin{enumerate}
\item For all $y\in V^*$, we define $\theta_y:(\g'_a)^*\longrightarrow (\g'_a)^*$ by $\theta_y(f)=yf$.
\item For all $x \in (\h_V)_+$, we put $\overline{\Delta}_\bullet(x)=\Delta_\bullet(x)-1\otimes x$ and $\overline{\Delta}_*(x)=\Delta_*(x)-1\otimes x$.
For all $g,f,f_1,\ldots,f_k \in (\g'_a)^*_+$:
$$\langle \overline{\Delta}_*(g),f\otimes f_1\ldots f_k\rangle=\langle g,f*f_1\ldots f_k\rangle.$$
\end{enumerate}

\subsection{Deshuffling coproducts}

\begin{proposition}
For all $g\in T(V)$, for all $i\in [N]$, $\Delta_{\shufd{j}}(g\epsilon_k)=\Delta_{\shuffle}(g)(\epsilon_k\otimes \epsilon_j)$.
\end{proposition}

\begin{proof} Let $f_1,f_2\in T(V)$, $i_1,i_2\in [N]$.
\begin{align*}
\langle \Delta_{\shufd{j}}(g\epsilon_k), f_1\epsilon_{i_1}\otimes f_2\epsilon_{i_2}\rangle
&=\langle g\epsilon_k, f_1\epsilon_{i_1} \shufd{j} f_2\epsilon_{i_2}\rangle\\
&=\delta_{i_2,j}\langle g\epsilon_k, f_1\shuffle f_2 \epsilon_{i_1}\rangle\\
&=\delta_{i_2,j}\delta_{i_1,k}\langle g,f_1\shuffle f_2\rangle\\
&=\delta_{i_2,j}\delta_{i_1,k}\langle \Delta_{\shuffle}(g),f_1\otimes f_2\rangle\\
&=\langle  \Delta_{\shuffle}(g)(\epsilon_k\otimes \epsilon_j),f_1\epsilon_{i_1}\otimes f_2\epsilon_{i_2}\rangle.
\end{align*}
As the pairing is nondegenerate, we obtain the result. \end{proof}\\

\textbf{Notations.} We define inductively, for $l\geq 0$, $j_1,\ldots,j_l \in [N]$:
$$\begin{cases}
\Delta_{\shuffle_\emptyset}=Id,\\
\Delta_{\shufd{j_1,\ldots,j_l}}=\left(\Delta_{\shufd{j_1}}\otimes Id^{\otimes(l-1)}\right)\circ \Delta_{\shufd{j_2,\ldots,j_l}}.
\end{cases}$$
For all $g\in T(V^*)$, for all $i\in [N]$:
$$\Delta_{\shufd{j_1,\ldots,j_l}}(g\epsilon_k)=\Delta_{\shuffle}^{(l)}(g)(\epsilon_k\otimes \epsilon_{j_1} \otimes \ldots \otimes \epsilon_{j_l});$$
for all $f_1,\ldots,f_l \in T(V)$:
$$\langle \Delta_{\shufd{j_1,\ldots,j_l}}(g),f_1\otimes \ldots \otimes f_{l+1}\rangle
=\langle g,f_1\shufd{j_1}\ldots \shufd{j_l} f_{l+1}\rangle.$$

\subsection{Dual of the post-Lie product}

\begin{proposition} \label{prop37}
In $\h_V=S((\g'_a)^*)/\langle \emptyset\epsilon_1\rangle$:
\begin{itemize}
\item For all $i\in [N]$, $\Delta_*(\emptyset\epsilon_i)=\emptyset\epsilon_i\otimes 1+1\otimes \emptyset\epsilon_i$.
\item For all $y \in V^*$, $g\in (\g'_a)^*$:
$$\overline{\Delta}_*\circ \theta_y(g)=\sum_{l\geq 0} \sum_{j_1,\ldots,j_l\in [N]} (\theta_{F_{j_1,\ldots,j_l}^*(y)}\otimes \mu)
\circ (\overline{\Delta}_*\otimes Id)\circ \Delta_{\shufd{j_1,\ldots,j_l}}(g),$$
where we denote by $\mu$ the sum of the iterated products of $\h_V$:
\begin{align*}
\mu:&\left\{\begin{array}{rcl}
T(\h_V)&\longrightarrow&\h_V\\
g_1\otimes\ldots \otimes g_k&\longrightarrow&g_1\ldots g_k.
\end{array}\right. \end{align*}
\end{itemize}\end{proposition}

\begin{proof} The first point comes from $\emptyset \epsilon_i*\U(\g'_a)_+=(0)$.\\

In order to prove the formula, it is enough to prove that, for $f,f_1,\ldots,f_k\in \g$:
$$\langle\overline{\Delta}_*\circ \theta_y(g), f\otimes f_1\ldots f_k\rangle
=\langle\sum_{l\geq 0} \sum_{j_1,\ldots,j_l\in [N]} (\theta_{F_{j_1,\ldots,j_l}^*(y)}\otimes \mu)
\circ (\overline{\Delta}_*\otimes Id)\circ \Delta_{\shufd{j_1,\ldots,j_l}}(g), f\otimes f_1\ldots f_k\rangle,$$
or equivalently:
$$\langle\theta_y(g), f*f_1\ldots f_k\rangle=\langle\sum_{l\geq 0} \sum_{j_1,\ldots,j_l\in [N]} (\theta_{F_{j_1,\ldots,j_l}^*(y)}\otimes \mu)
\circ (\overline{\Delta}_*\otimes Id)\circ \Delta_{\shufd{j_1,\ldots,j_l}}(g), f\otimes f_1\ldots f_k\rangle,$$
If $f=\emptyset\epsilon_i$, both sides are equal to $0$. Otherwise, we can assume that $f=xf'$, with $x\in V$ and $f'\in \g$.
\begin{align*}
&\langle\theta_y(g), f*f_1\ldots f_k\rangle\\
&=\langle yg, \sum_{I=\{i_1<\ldots<i_l\}\subseteq [k]}\sum_{j_1,\ldots,j_l\in [N]} 
F_{j_1,\ldots,f_l}(x)\left(f'*\left(\prod_{i\notin I}f_i\right)\shufd{j_1}f_{i_1}\ldots \shufd{j_l} f_{i_l}\right)\rangle\\
&=\sum_{I=\{i_1<\ldots<i_l\}\subseteq [k]}\sum_{j_1,\ldots,j_l\in [N]} \langle y,F_{j_1,\ldots,j_l}(x)\rangle
\langle \Delta_{\shufd{j_1,\ldots,j_l}}(g),f'*\left(\prod_{i\notin I}f_i\right)\otimes f_{i_1}\ldots \otimes f_{i_l}\rangle\\
&=\sum_{j_1,\ldots,j_l\in [N]} \langle F_{j_1,\ldots,j_l}^*(y),x\rangle
\langle (Id \otimes \mu)\circ (\overline{\Delta}_*\otimes Id)\circ \Delta_{\shufd{j_1,\ldots,j_l}}(g),f'\otimes f_1\ldots f_k\rangle\\
&=\sum_{j_1,\ldots,j_l\in [N]}
\langle (\theta_{F_{j_1,\ldots,j_l}^*(y)}\otimes Id)\circ (Id \otimes \mu)\circ (\overline{\Delta}_*\otimes Id)\circ \Delta_{\shufd{j_1,\ldots,j_l}}(g),
xf'\otimes f_1\ldots f_k\rangle,
\end{align*}
which ends the proof. \end{proof}\\

In order to obtain a better description of the coproduct $\overline{\Delta}_*$, we are going to identify the following three objects:
$$\xymatrix{&S((\g'_a)_+^*) \ar[rd]^{\sim} \ar[ld]_{\sim}&\\
S((\g'_a)^*)/\langle \emptyset \epsilon_1\rangle&&S((\g'_a)^*)/\langle \emptyset \epsilon_1-1\rangle}$$
Both identification sends $x\in(\g'_a)^*_+$ to its class. Let us  reformulate Proposition \ref{prop37}
in the vector space $S((\g'_a)^*)/\langle \emptyset \epsilon_1-1\rangle$:
\begin{align*}
\overline{\Delta}_*\circ \theta_y(g\epsilon_k)&=\sum_{l\geq 0} \sum_{j_1,\ldots,j_l\in [N]} (\theta_{F_{j_1,\ldots,j_l}^*(y)}\otimes \mu)
\circ (\overline{\Delta}_*\otimes Id)(\Delta_{\shuffle}^{(l)}(g)\epsilon_k\otimes \epsilon_{j_1}\otimes \ldots \otimes \epsilon_{j_l})\\
&-\sum_{l\geq 0} \sum_{j_1,\ldots,j_l\in [N]} (\theta_{F_{j_1,\ldots,j_l,1}^*(y)}\otimes \mu)
\circ (\overline{\Delta}_*\otimes Id)(\Delta_{\shuffle}^{(l+1)}(g)\epsilon_k\otimes \epsilon_{j_1}\otimes \ldots \otimes \epsilon_{j_l}\otimes \epsilon_1)\\
&=\sum_{l\geq 0} \sum_{j_1,\ldots,j_l\in [N]} (\theta_{F_{j_1,\ldots,j_l}^*(y)}\otimes \mu)
\circ (\overline{\Delta}_*\otimes Id)(\Delta_{\shuffle}^{(l)}(g)\epsilon_k\otimes \epsilon_{j_1}\otimes \ldots \otimes \epsilon_{j_l})\\
&-\left(\sum_{l\geq 0} \sum_{j_1,\ldots,j_l\in [N]} (\theta_{F_{j_1,\ldots,j_l,1}^*(y)}\otimes \mu)
\circ (\overline{\Delta}_*\otimes Id)(\Delta_{\shuffle}^{(l)}(g)\epsilon_k\otimes \epsilon_{j_1}\otimes \ldots \otimes \epsilon_{j_l})\right)
(1\otimes \emptyset\epsilon_1).
\end{align*}
Finally, identifying in $S((\g'_a)_+^*)$:

\begin{proposition}
For all $j_1,\ldots,j_l\in [N]$, we put:
$$G_{j_1,\ldots,j_l}=F_{j_1,\ldots,j_l}-F_{j_1,\ldots,j_l,1}.$$
In $S((\g'_a)^*_+)/\langle \emptyset \epsilon_1-1\rangle$:
\begin{itemize}
\item For all $i\in [N]$, $\overline{\Delta}_*(\emptyset \epsilon_i)=\emptyset\epsilon_i\otimes 1$.
\item For all $y\in V^*$, for all $g\in(\g'_a)^*_+$:
$$\overline{\Delta}_*\circ \theta_y(g)=\sum_{l\geq 0} \sum_{j_1,\ldots,j_l\in [N]} (\theta_{G_{j_1,\ldots,j_l}^*(y)}\otimes \mu)
\circ (\overline{\Delta}_*\otimes Id)\circ \Delta_{\shufd{j_1,\ldots,j_l}}(g).$$
\end{itemize}\end{proposition}

\textbf{Example.} For $\siso$, as $V$ is a module over the associative algebra $(\g_{(1,0)},\triangleleft)$, if $l\geq 2$, $F_{j_1,\ldots,j_l}=0$
by Proposition \ref{prop25}, so  $G_{j_1,\ldots,j_l}=0$. Moreover:
\begin{align*}
F_\emptyset&=\left(\begin{array}{cc}
1&0\\ 0&1
\end{array}\right),&
F_1&=\left(\begin{array}{cc}
0&0\\ 0&1
\end{array}\right),&
F_2&=\left(\begin{array}{cc}
0&1\\ 0&0
\end{array}\right).\\
G_\emptyset=F_\emptyset-F_1&=\left(\begin{array}{cc}
1&0\\ 0&0
\end{array}\right),&
G_1=F_1&=\left(\begin{array}{cc}
0&0\\ 0&1
\end{array}\right),&
G_2=F_2&=\left(\begin{array}{cc}
0&1\\ 0&0
\end{array}\right).\\
G_\emptyset^*&=\left(\begin{array}{cc}
1&0\\ 0&0
\end{array}\right),&
G_1^*&=\left(\begin{array}{cc}
0&0\\ 0&1
\end{array}\right),&
G_2^*&=\left(\begin{array}{cc}
0&0\\1&0
\end{array}\right).
\end{align*}
The coproduct $\overline{\Delta}_*$ on $S((\siso)^*_+)$ is given by:
\begin{itemize}
\item For all $i\in [2]$, $\overline{\Delta}_*(\emptyset\epsilon_i)=\emptyset\epsilon_i\otimes 1$.
\item For all $g\in \K\langle x_1,x_2\rangle$, for all $i\in [2]$:
\begin{align*}
\overline{\Delta}_*\circ \theta_{x_1}(g\epsilon_i)&=(\theta_{x_1}\otimes Id)\circ \overline{\Delta}_*(g\epsilon_i)+
(\theta_{x_2}\otimes \mu)\circ(\overline{\Delta}_*\otimes Id)(\Delta_{\shuffle}(g) \epsilon_i\otimes \epsilon_2),\\
\overline{\Delta}_*\circ \theta_{x_2}(g\epsilon_i)&=(\theta_{x_2}\otimes \mu)\circ(\overline{\Delta}_*\otimes Id)(\Delta_{\shuffle}(g) \epsilon_i\otimes \epsilon_1).
\end{align*}
\end{itemize}
These are formulas of Lemma 4.1 of \cite{Gray}, where $a_w=w\epsilon_2$, $b_w=w\epsilon_1$, $\theta_0=\theta_{x_1}$,
$\theta_1=\theta_{x_2}$ and $\tilde{\Delta}=\overline{\Delta}_*$.

\subsection{Dual of the pre-Lie product}

\textbf{Notations.} We denote by $\Delta_{\shufg{1}}$ the coproduct on $T_+(V^*)\otimes (V)^{N-1}$ dual to the product $\shufg{1}$.
As $\shufg{1}=\shufd1^{op}$, $\Delta_{\shufg{1}}=\Delta_{\shufd1}^{cop}$, and for all $g\in T(V)$, for all $i\in [N]$:
$$\Delta_{\shufg{1}}(g\epsilon_i)=\Delta_{\shuffle}(g)(\epsilon_1\otimes \epsilon_k).$$

\begin{proposition}
In $S((\g'_a)_+^*)/\langle \emptyset\epsilon_1\rangle$, for all $g\in(\g'_a)^*_+$:
\begin{align*}
\overline{\Delta}_\bullet(g)&=\overline{\Delta}_*(g)+(Id \otimes \mu)\circ (\overline{\Delta}_*\otimes Id)\circ \Delta_{\shufg{1}}(g).
\end{align*}\end{proposition}

\begin{proof} Let $f,f_1,\ldots,f_k \in (\g'_a)_+$.
\begin{align*}
\langle \overline{\Delta}_\bullet(g), f\otimes f_1\ldots f_k \rangle
&=\langle g, f\bullet f_1\ldots f_k\rangle\\
&=\langle g,f*f_1\ldots f_k+\sum_{p=1}^k (f*f_1\ldots \widehat{f_p}\ldots f_k) \shufg{1} f_p\rangle\\
&=\langle \overline{\Delta}_*(g),f\otimes f_1\ldots f_k\rangle+\langle \Delta_{\shufg{1}}(g),\sum_{p=1}^k f*f_1\ldots \widehat{f_p}\ldots f_k\otimes f_p\rangle\\
&=\langle \overline{\Delta}_*(g),f\otimes f_1\ldots f_k\rangle+\langle (\Delta_*\otimes Id) \circ\Delta_{\shufg{1}}(g),\sum_{p=1}^k f\otimes 
f_1\ldots \widehat{f_p}\ldots f_k\otimes f_p\rangle\\
&=\langle \overline{\Delta}_*(g),f\otimes f_1\ldots f_k\rangle+\langle (Id \otimes \mu)
\circ(\Delta_*\otimes Id)\circ \Delta_{\shufg{1}}(g),f\otimes f_1\ldots f_k\rangle.
\end{align*}
As $(\g'_a,*)$ is pre-Lie, $\overline{\Delta}_\bullet(g) \in(\g'_a)^*_+\otimes S((\g'_a)_+^*)$ and the nondegeneracy of the pairing 
implies the formula. \end{proof}\\

Rewriting this formula in $S((\g'_a)_+^*)/\langle \emptyset\epsilon_1-1\rangle$:
\begin{align*}
\overline{\Delta}_\bullet(g\epsilon_1)&=\overline{\Delta}_*(g\epsilon_1)+(Id \otimes \mu)\circ (\overline{\Delta}_*\otimes Id)
(\Delta_{\shuffle}(g)(\epsilon_1\otimes \epsilon_1))\\
&=\overline{\Delta}_*(g\epsilon_1)+(Id \otimes \mu)\circ (\overline{\Delta}_*\otimes Id)
((\Delta_{\shuffle}(g)-g\otimes \emptyset)(\epsilon_1\otimes \epsilon_1))\\
&=\overline{\Delta}_*(g\epsilon_1)(1\otimes (1-\emptyset\epsilon_1))+(Id \otimes \mu)\circ (\overline{\Delta}_*\otimes Id)
(\Delta_{\shuffle}(g)(\epsilon_1\otimes \epsilon_1))\\
&=(Id \otimes \mu)\circ (\overline{\Delta}_*\otimes Id)(\Delta_{\shuffle}(g)(\epsilon_1\otimes \epsilon_1)).
\end{align*}

Identifying in $S((\g'_a)_+^*)$:

\begin{proposition}
In $S((\g'_a)^*_+)/\langle \emptyset \epsilon_1-1\rangle$, if $g\in T(V^*)$:
\begin{align*}
\overline{\Delta}_\bullet(g\epsilon_1)&=(Id\otimes \mu)\circ (\overline{\Delta}_*\otimes Id)(\Delta_{\shuffle}(g)(\epsilon_1\otimes \epsilon_1)),\\
\mbox{ if }i\geq 2,\: \overline{\Delta}_\bullet(g\epsilon_i)&=\overline{\Delta}_*(g\epsilon_i)+
(Id\otimes \mu)\circ (\overline{\Delta}_*\otimes Id)(\Delta_{\shuffle}(g)(\epsilon_i\otimes \epsilon_1)),
\end{align*}
with the convention $\emptyset\epsilon_1=1$. We put $\Delta_\bullet(g)=\overline{\Delta}_\bullet(g)+1\otimes g$ for all $g\in (\g_a')^*_+$
and extend $\Delta_\bullet$ to $S((\g'_a)_+^*)$ as an algebra morphism. 
This coproduct makes $S((\g'_a)_+^*)$ a Hopf algebra, isomorphic to the graded dual of the enveloping algebra of $((\g'_a)_+,[-,-]_*)$.
\end{proposition}

\textbf{Remark.} These are \emph{mutatis mutandis} the formulas of Lemma 4.3 in \cite{Gray}.

\bibliographystyle{amsplain}
\bibliography{biblio}

\end{document}